\documentclass[12pt,a4paper]{article}

\usepackage[utf8]{inputenc}

\usepackage{rsfso}
\usepackage{mathtools,amsfonts,amssymb,mathrsfs}
\usepackage{enumerate}
\usepackage{lmodern}
\usepackage[T1]{fontenc}
\usepackage{verbatim}
\usepackage{textcomp}
\usepackage[a4paper,vmargin={3.5cm,3.5cm},hmargin={2.5cm,2.5cm}]{geometry}
\usepackage[font=sf, labelfont={sf,bf}, margin=1cm]{caption}

\usepackage[pdftex]{color,graphicx}
\usepackage{pdfpages}

\usepackage{stmaryrd}

\usepackage{xspace}

\usepackage{dsfont}

\usepackage{enumitem}

\usepackage{amsthm}

\usepackage{color}

\usepackage[english]{babel}
\usepackage[pdftex]{hyperref}

\newcommand\ud{\textrm{d}}

\newcommand\kH{\mathfrak{H}}

\newcommand\E{\mathbb{E}}
\newcommand\F{\mathbb{F}}
\renewcommand\P{\mathbb{P}}		
\newcommand\Q{\mathbb{Q}}
\newcommand\T{\mathbb{T}}
\newcommand\Z{\mathbb{Z}}

\renewcommand\AA{\mathcal{A}}		
\newcommand\FF{\mathcal{F}}
\newcommand\HH{\mathcal{H}}
\newcommand\KK{\mathcal{K}}

\newcommand{\vd}{\delta}
\newcommand{\ve}{\varepsilon}
\newcommand{\vk}{\kappa}
\newcommand{\vl}{\lambda}
\newcommand{\vD}{\Delta}
\newcommand{\vT}{\Theta}
\newcommand{\vg}{\gamma}

\renewcommand{\l}{\left}
\renewcommand{\r}{\right}

\newcommand{\eqdef}{\overset{\text{def}}{=}}

\newcommand{\card}[1]{\left| #1 \right|}

\newcommand{\ceil}[1]{{\left\lceil #1 \right\rceil}}
\newcommand{\usim}[2]{\underset{#1 \to #2}{\sim}}

\newcommand{\ulim}[2]{\underset{#1 \to #2}{\longrightarrow}}

\newcommand{\pfrac}[2]{{\left(\frac{#1}{#2}\right)}}

\newcommand{\figref}[1]{figure \ref{#1}}

\newcommand{\dgr}{\mathrm{d}_{\mathrm{gr}}}

\newcommand{\Htr}{\HH^\text{tr}}
\newcommand{\Hull}{B^\bullet}

\newcommand{\UIPQ}{\textsc{uipq}\xspace}

\newcommand{\BGW}{Bienaymé-Galton-Watson\xspace}

\newcommand{\resp}{resp.\xspace}

\newtheorem{theorem}{Theorem}
\newtheorem{lemma}[theorem]{Lemma}

\newtheorem{proposition}[theorem]{Proposition}

\newtheorem{corollary}[theorem]{Corollary}

\newtheorem{conjecture}[theorem]{Conjecture}

\newcommand{\propref}{Proposition \ref}
\newcommand{\thref}{Theorem \ref}
\newcommand{\corref}{Corollary \ref}
\newcommand{\lemref}{Lemma \ref}

\renewcommand{\figref}{Figure \ref}

\newenvironment{romenum}{
\begin{enumerate}[label=(\roman*)]
}{
\end{enumerate}
}

\newcommand{\Qtr}{\Q^{\text{tr}}}

\DeclareMathOperator{\arccosh}{cosh^{-1}}
\renewcommand{\card}[1]{\#{#1}}

\newcounter{cstes}

\newcommand{\UIPT}{\textsc{uipt}\xspace}

\newcommand{\mult}{\operatorname{mult}}
\newcommand{\CVS}{\textsc{cvs} bijection\xspace}

\author{Thomas \textsc{Lehéricy}}
\title{Uniform mixing time and bottlenecks in uniform finite quadrangulations}

\begin{document}

\maketitle

\begin{abstract}
We prove a lower bound on the size of bottlenecks in uniform quadrangulations, valid at all scales simultaneously. 
We use it to establish upper bounds on the uniform mixing time of the lazy random walk on uniform quadrangulations, as well as on their dual. 
The proofs involve an explicit computation of the Laplace transform of the number of faces in truncated hulls of the uniform infinite plane quadrangulation. 
\end{abstract}

\tableofcontents

\section{Introduction}

\paragraph{Uniform mixing time in uniform quadrangulations.}

A \emph{rooted planar map} is an embedding of a planar graph on the sphere with no edge-crossing, seen up to orientation-preserving homeomorphisms, equipped with a distinguished oriented edge called the \emph{root edge}. A quadrangulation is a rooted planar map such that all its faces have degree 4. In this paper, we will be interested in type I quadrangulations, where multiple edges are allowed.

If $Q$ is a quadrangulation, denote its vertex set, \resp edge set, face set, by $V(Q)$, \resp $E(Q), F(Q)$, and write $|Q| \eqdef |F(Q)|$ for its number of faces. The degree of a vertex $x$ of $Q$ is denoted by $\deg_Q(x)$, and the number of edges with endpoints $x$ and $y$ is denoted by $\mult_Q(x,y)$. 
We are interested in the lazy random walk on $Q$, which is a reversible Markov chain on $V(Q)$ with transition probabilities 
\begin{equation*}
p_Q(x,y) = \begin{cases}
               1/2 &\text{ if }x=y , \\
               \frac{\mult_Q(x,y)}{2\deg_Q(x)} &\text{ if }x \neq y ,
            \end{cases}
\end{equation*}
and stationary distribution 
\begin{equation}
\pi_Q(x) = \frac{\deg_Q(x)}{4|Q|} .
\end{equation}
The choice of the lazy random walk over the simple random walk is technical in nature. In particular, the lazy random walk is aperiodic even when the simple random walk is not. we fully expect our results to still hold for the simple random walk, provided the non-aperiodicity of the walk on the (bipartite) quadrangulations is properly handled. 

One may check that $\sum_{y\in V(Q)}\deg_Q(y) = 2|E(Q)| = 4|Q|$, so $\pi_Q$ is a probability distribution. 
We write $p^k_Q(x,y)$ for the $k$-step transition probabilities of the lazy random walk and define the $\ve$-uniform mixing time of the lazy random walk
\begin{equation*}
\tau_Q(\ve) \eqdef \inf \l\{  k : \forall x,y\in V(Q), \  \l|\frac{p_Q^k(x,y)-\pi_Q(y)}{\pi_Q(y)} \r| \leq \ve \r\} .
\end{equation*}

Let $Q_n$ be a uniform quadrangulation with $n$ faces. Our first theorem provides an upper bound on the mixing time of the lazy random walk in $Q_n$.

\begin{theorem}
\label{Th_mixing_time}
For every $\ve,\vd>0$, with probability going to $1$ as $n \to \infty$,
\begin{equation*}
\tau_{Q_n}(\ve) \leq n^{3/2+\vd} .
\end{equation*}
\end{theorem}

Our bound relies on a known result that relates the uniform mixing time of the lazy random walk on a graph to the size of ``bottlenecks'' \cite{jerrum1989approximating}, i.e. small sets that separate the graph into two large connected components. 
The narrower the bottlenecks, the harder it is for the random walk to cross them, and the longer the mixing time. Conversely, if there is no very narrow bottleneck then the mixing time will not be too large. 
 The bulk of this article is thus dedicated to showing that bottlenecks cannot be too narrow, see Theorem \ref{Th_isoperimetric_inequality} and Corollary \ref{Cor_cheeger}.

We could derive from Corollary \ref{Cor_cheeger} and \cite[Theorem 7.4]{levin2017markov} a lower bound of the form $n^{3/4+o(1)}$ on the mixing time in total variation for the lazy random walk. However, $3/4$ is in all likelihood not the optimal exponent: \cite{gwynne2020anomalous} proved that the simple random walk on the \UIPT (the local limit of uniform triangulations as their size goes to infinity) travels a distance $t^{1/4+o(1)}$ after time $t$. We thus expect that the mixing time should be of order at least $n^{1+o(1)}$. Indeed, roughly speaking, since a map with $n$ faces has diameter $n^{1/4+o(1)}$, we need to wait for a time $n^{1+o(1)}$ before the random walk has a chance to explore the whole map.

If $Q$ is a quadrangulation, the dual $Q^\dagger$ of $Q$ is the planar graph whose vertices are the faces of $Q$, where two faces of $Q$ are adjacent if they share an edge in $Q$. 
We prove a similar upper bound on the uniform mixing time of the lazy random walk on $Q^\dagger$, which is a reversible Markov chain on $F(Q)$ with the following transition probability: at each time step, the walk has probability $1/2$ of staying at the same face, and probability $1/2$ of crossing one of the four sides of the current face, chosen uniformly at random. Note that $Q$ is of type I, so both sides of a given edge may be incident to the same face; crossing such an edge results in staying at the current face. We denote the transition kernel of the lazy random walk by $p_{Q^\dagger}$. 
Its stationary distribution $\pi_{Q^\dagger}$ is the uniform probability measure on $F(Q)$. The $\ve$-uniform mixing time of the lazy random walk is defined as before:
\begin{equation*}
\tau_{Q^\dagger}(\ve) \eqdef \inf \l\{  k : \forall x,y\in F(Q), \  \l|\frac{p_{Q^\dagger}^k(x,y)-\pi_{Q^\dagger}(y)}{\pi_{Q^\dagger}(y)} \r| \leq \ve \r\} .
\end{equation*}

Our second theorem provides an upper bound on the mixing time of the lazy random walk in the dual of $Q_n$.

\begin{theorem}
\label{Th_mixing_time_dual}
For every $\ve,\vd>0$, with probability going to $1$ as $n \to \infty$,
\begin{equation*}
\tau_{Q^\dagger_n}(\ve) \leq n^{3/2+\vd} .
\end{equation*}
\end{theorem}

\paragraph{Bottlenecks in finite quadrangulations.}

We now state our lower bounds on the size of bottlenecks in $Q_n$. The first bound considers sets of faces of $Q_n$. 
 For every $S \subset F(Q_n)$, we denote the set of all edges of $Q_n$ incident on one side to a face of $S$ and on the other side to a face outside of $S$ by $\partial S$. 
\begin{theorem}
\label{Th_isoperimetric_inequality}
For every $\nu \in (0,1)$:
\begin{equation*}
\P\l( \inf_{S\subset F(Q_n)\ :\ 0<|S| \leq n/2} \frac{|\partial S|^{4/3}}{|S|} \geq n^{-2/3-\nu} \r) \ulim n \infty 1.
\end{equation*}
\end{theorem}

This theorem is instrumental in the proof of Theorems \ref{Th_mixing_time}, \ref{Th_mixing_time_dual} and \ref{Th_isoperimetric_inequality_V}. Section 3 to 6 are devoted to its proof.

An interesting feature of \thref{Th_isoperimetric_inequality} is that the bound holds for all scales simultaneously: $S$ can have any size, and is not restricted to contain a macroscopic fraction of faces of $Q_n$. We conjecture that the bound of the Theorem is the best possible, in the sense that for every $0 < k < n/4$, we can find an $S$ with $k \leq |S| \leq 2k$ and $\frac{|\partial S|^{4/3}}{|S|} \approx n^{-2/3}$.

In order to establish our results, we will heavily study the local limit of quadrangulations, the uniform infinite plane quadrangulation or \UIPQ \cite{Krikun2008local}.

Let us mention earlier results in this direction. 
\cite{gall2017separating} established a lower bound on the size of bottlenecks in the uniform infinite plane quadrangulation or \UIPQ, in the form of an isoperimetric inequality. More precisely, \cite[Theorem 3]{gall2017separating} ensures that any connected union of $n$ faces of the \UIPQ, such that at least one of the faces is incident to the root vertex, has a boundary that must contain at least $n^{1/4}(\log n)^{-(3/4)-\vd}$ edges. However, this result is not sufficient for our purpose: firstly because it applies to the infinite-volume limit of uniform quadrangulations, secondly because it only controls the size of bottlenecks that separate the root vertex from infinity. Our results are established independently from those in \cite{gall2017separating}.

The convergence of uniform random quadrangulations towards the Brownian map \cite{leGall2013uniqueness,miermont2013brownian} gives a rough lower bound on the size of macroscopic bottlenecks in finite quadrangulations. 
Let us be more precise. Fix $\vd>0$. Since the Brownian map is homeomorphic to the sphere \cite{paulinLG2008sphere}, we can find $\ve>0$ such that with probability close to $1$, for $n$ large enough, any cycle in $Q_n$ that separates $Q_n$ in two subsets, each with at least $\vd n$ faces, must have length at least $\ve n^{1/4}$. This is the best one can expect: with high probability it is possible to find sets of size roughly $n/2$ and perimeter no larger than some large constant times $n^{1/4}$. 
However, this result only gives a lower bound on the size of bottlenecks at large scales (where the infimum holds over subsets $S$ of $F(Q_n)$ with $\vd n \leq |S| \leq n/2$).

Let us give the intuition why this bound is optimal, focusing on large and small scales only. 
Our previous remark ensures that $\inf_{S\subset F(Q_n)\ :\ \vd n <|S| \leq n/2} \frac{|\partial S|^{4/3}}{|S|} \approx n^{-2/3+o(1)}$, so the bound of \thref{Th_isoperimetric_inequality} is indeed optimal for sets containing a macroscopic proportion of faces of $Q_n$.  
For small scales, \cite[Proposition 5]{banderier2001random} states that the supremum over all cycles of length $2$ of the number of faces contained in the smallest component of the complement of the cycle is at most $n^{2/3+o(1)}$. We expect the supremum to be indeed $n^{2/3+o(1)}$; if this is true, then the bound of \thref{Th_isoperimetric_inequality} is also optimal for sets containing at most $n^{2/3+o(1)}$ faces.

We now state our second bound on the size of bottlenecks $Q_n$, which holds for sets of vertices of $Q_n$. For every $A,B \subset V(Q_n)$, we denote the set of all edges of $Q_n$ with one endpoint in $A$ and the other in $B$ by $E(A,B)$, and $A^c \eqdef V(Q_n) \setminus A$. 
\begin{theorem}
\label{Th_isoperimetric_inequality_V}
For every $\nu \in (0,1)$:
\begin{equation*}
\P\l( \inf_{A\subset V(Q_n)\ :\ \pi_{Q_n}(A) \leq 1/2} \frac{|E(A,A^c)|^{4/3}}{4n\pi_{Q_n}(A)} \geq n^{-2/3-\nu} \r) \ulim n \infty 1.
\end{equation*}
\end{theorem}

The similarity between Theorem \ref{Th_isoperimetric_inequality} and Theorem \ref{Th_isoperimetric_inequality_V} can be highlighted by noticing that $E(A,A^c)$ is the set of all edges with one end in $A$ and another in $A^c$, while $\partial S$ is the set of all edges with one side in $S$ and another in $S^c$. The denominator in Theorem \ref{Th_isoperimetric_inequality} may also be rewritten in a way closer to Theorem \ref{Th_isoperimetric_inequality_V}: $|S| = n \pi_{Q_n^\dagger}(S)$.

\paragraph{Hull volume.}

For every integer $r>0$ and every vertex $v$ of the \UIPQ, the $r$-ball of the \UIPQ $Q_\infty$ centered at $v$ is the union of faces of the \UIPQ that are incident to a vertex at distance at most $r-1$ from $v$. The standard $r$-hull centered at $v$, denoted by $\Hull_{Q_\infty}(v,r)$ is the union of the ball and of the finite connected components of its complement. 

In order to prove Theorem \ref{Th_isoperimetric_inequality}, we compute the Laplace transform of the volume of truncated hulls centered at the root vertex of the \UIPQ, and derive an upper bound on the probability that their volume is large. Truncated hulls is a type of hulls that is particularly adapted to the decomposition of the \UIPQ into layers, see \cite{gall2017separating,lehe2019fpp} and Section \ref{Sec_Laplace_transform_hull_volume} for a precise definition. They are also closely related to the standard hulls in the following way: if $\Htr_{Q_\infty}(v,r)$ is the truncated $r$-hull of the \UIPQ centered at $v$, then the following inclusions are verified for every integer $r>0$:
\begin{equation}
\label{Eq_inclusion_hulls}
\Htr_{Q_\infty}(v,r) \subset \Hull_{Q_\infty}(v,r) \subset \Htr_{Q_\infty}(v,r+1) .
\end{equation}
It follows that bounds on the volume can be transfered from truncated hulls to standard hulls and vice versa, up to negligible terms. 
Our methods are reminiscent of the paper \cite{menard2018volumes} dealing with triangulations, although the formulas are more involved. To keep the technicalities to a minimum, we only establish an analog of Theorem 1 in \cite{menard2018volumes}.

 Let us state an interesting consequence of our computations, which is a key tool in Section 5. 
\begin{lemma}
For every $\vd>0$, there exists $C(\vd) \in (0,\infty)$ such that for every $r>0$, for every $t>0$,
\begin{equation*}
\P\l( |\Hull_{Q_\infty}(r)| > t r^4 \r) \leq C(\vd) t^{-3/2+\vd} .
\end{equation*}
\end{lemma}
This bound is not surprising; the volume of the unit hull of the Brownian Plane has a known distribution that exhibits such a tail, without the $\vd$ in the exponent. In fact, we also obtain an asymptotic of the tail of $|\Htr_{Q_\infty}(r)|$ as $t \to \infty$:
\begin{equation*}
\P\l( |\Htr_{Q_\infty}(r)| > t \r) \usim t \infty \frac{r(r+3)(r+1)^3(r+2)^3}{4 \sqrt \pi(2r+3)^2} t^{-3/2} .
\end{equation*}
However, the tails are not useful directly: for our purposes we need a tail estimate that is valid for every $t>0$ (and not just as $t\to\infty$).

\paragraph{Plan of the paper.}

We explain in Section \ref{Sec_Proof of the mixing time theorem} how to derive the bound on the mixing time in Theorems \ref{Th_mixing_time},  \ref{Th_mixing_time_dual} and \ref{Th_isoperimetric_inequality_V}, from the bound on the size of bottlenecks in Theorem \ref{Th_isoperimetric_inequality}. The rest of the article is devoted to the proof of Theorem \ref{Th_isoperimetric_inequality}. 
At the core of the proof lies a fine control of the volume of standard hulls, more precisely of the probability that their volume is large. Such a control is easier to establish in the \UIPQ. Section \ref{Sec_Standard hulls in finite quadrangulations and density with the UIPQ} derives Lemma \ref{Lemma_density_hulls_UIPQ_Qn}, which allows us to transfer results from the \UIPQ to finite quadrangulations. 
We compute in Section \ref{Sec_Laplace_transform_hull_volume} the Laplace transform of the volume of truncated hulls in the \UIPQ, and derive the required tail estimates from its Taylor expansion near $0$. This section makes heavy use of the so-called skeleton decomposition of the \UIPQ.
 Section \ref{Section_Covering} uses the two previous sections to establish Proposition \ref{Prop_covering_hulls_with_controlled_volume}, stating that we can cover a uniform finite quadrangulation with a “small” number of hulls whose volumes are “controlled”. Finally, we prove Theorem \ref{Th_isoperimetric_inequality} in Section \ref{Sec_Isoperimetric inequality}.

\section{Proof of the mixing time theorem}
\label{Sec_Proof of the mixing time theorem}

We derive \thref{Th_mixing_time_dual} from \thref{Th_isoperimetric_inequality}. 
The first step is a bound on the Cheeger constant of $Q_n^\dagger$, which is a straightforward (and much weaker) consequence of \thref{Th_isoperimetric_inequality}.
\begin{corollary}
\label{Cor_cheeger}
For every $\vd>0$, the Cheeger constant $$\kH^\dagger \eqdef \inf \l\{ \frac{|\partial S|}{|S|} \ : \ S\subset F(Q_n), \ 0<|S|\leq n/2 \r\} $$ of $Q_n^\dagger$ is larger than $n^{-3/4-\vd}$ with probability going to $1$ as $n \to \infty$.
\end{corollary}

\begin{proof}
If $0<|S|\leq n/2$, 
\begin{equation*}
\frac{|\partial S|}{|S|} \geq \l( \frac{|\partial S|^{4/3}}{|S|} \r)^{3/4} \pfrac{2}{n}^{1/4}
\end{equation*}
and we just need to apply \thref{Th_isoperimetric_inequality}.
\end{proof}

\begin{proof}[Proof of \thref{Th_mixing_time_dual}]

For every $x,y \in F(Q_n)$, let $w_{x,y} \eqdef p_{Q_n^\dagger}(x,y) \pi_{Q_n^\dagger}(x)$. The conductance of $Q_n^\dagger$ is
\begin{equation}
\Phi^\dagger \eqdef \inf\l\{ \frac{\sum_{x \in S, y\notin S} w_{x,y}}{\pi_{Q_n^\dagger}(S)} \ : \ S \subset F(Q_n), \ \pi_{Q_n^\dagger}(S) \leq 1/2 \r\} .
\end{equation}
For every adjacent and distinct $x,y \in F(Q_n)$, $p_{Q_n^\dagger}(x,y)\geq 1/8$ and $w_{x,y} \geq 1/8n$, thus 
$$\sum_{x \in S, y\notin S} w_{x,y} \geq \frac{|\partial S|}{8n} ,$$
and $\Phi^\dagger \geq \kH^\dagger/8 $. We then apply \cite[Corollary 2.3]{jerrum1989approximating}: 
\begin{equation}
\tau_{Q_n^\dagger}(\ve) \leq \frac{128}{{\kH^\dagger}^2}(\ln(8n) + \ln(1/\ve) ) .
\end{equation}
The theorem then follows from \corref{Cor_cheeger}.

\end{proof}

We now derive \thref{Th_isoperimetric_inequality_V} and \thref{Th_mixing_time} from \thref{Th_isoperimetric_inequality}. Let us first prove an upper bound on the maximum degree of a vertex of $Q_n$.

\begin{lemma}
\label{Lemma_max_deg_Qn}
Let $\vD_n \eqdef \max \{ \deg_{Q_n}(x) \ : \ x \in V(Q_n) \}$. Then
\begin{equation}
\P(\vD_n > \ln n) \ulim n \infty 0 .
\end{equation}
\end{lemma}

\begin{proof}
We recall the so-called ``trivial'' bijection between the set of all quadrangulations with $n$ faces and the set of all rooted planar maps with $n$ edges: let $Q$ be a quadrangulation with $n$ faces, and color its vertices so that the tail of the root vertex is white, and every two adjacent vertices have different colors. In each face of $Q$, draw a diagonal between the two white corners of the face. The map obtained by keeping only the added diagonals, together with the white vertices of $Q$, is a planar map $M$ with $n$ edges, that we root at the edge contained in the root face of $Q$ in such a way that the root vertex is the same as that of $Q$.

The set of all white vertices of $Q$ is exactly the set of all vertices of $M$, and every face of $M$ contains exactly one black vertex of $Q$. An easy observation is that, if $x$ is a white vertex of $Q$, then $\deg_Q(x) = \deg_M(x)$. If $x$ is a black vertex of $Q$, then $\deg_Q(x)$ is the degree of the face of $M$ that contains $x$. 

Denote the image of $Q_n$ under the trivial bijection by $M_n$. $M_n$ is uniformly distributed over the set of all rooted planar maps with $n$ edges. \cite[Theorem 3]{gao2000distribution} ensures that, writing $\vD^{M_n}$ for the maximum degree of a vertex of $M_n$,
\begin{equation*}
\P(\vD^{M_n} > \ln n) \ulim n \infty 0 .
\end{equation*}
By self-duality of $M_n$, the same holds when replacing $\vD^{M_n}$ by the maximum degree of a face of $M_n$. The lemma follows by our previous observation.
\end{proof}

\begin{proof}[Proof of Theorem \ref{Th_isoperimetric_inequality_V}.]

Let $\nu>0$. Fix $N(\nu)$ so that $8 n^{-\nu/3} \ln n < 1$ for every $n \geq N(\nu)$. 
Let $Q$ be a quadrangulation of size $n \geq N(\nu)$ such that
\begin{align}
\label{Eq_hyp_ineg_isop}
\inf_{S\subset F(Q)\ :\ 0<|S| \leq n/2} \frac{|\partial S|^{4/3}}{|S|} &\geq n^{-2/3-\nu} , \\
\label{Eq_max_degre_Q}
\max_{x \in V(Q)} \deg_{Q}(x) &\leq \ln n .
\end{align}
Recall that the probability that $Q_n$ satisfies \eqref{Eq_hyp_ineg_isop} goes to $1$ as $n \to \infty$ by Theorem \ref{Th_isoperimetric_inequality}, and the probability that $Q_n$ satisfies \eqref{Eq_max_degre_Q} goes to $1$ as well by \lemref{Lemma_max_deg_Qn}.

For every $A \subset V(Q)$, we let $E = E(A,A^c)$ be the set of all edges of $Q$ with one end in $A$ and one end in $A^c = V(Q) \setminus A$. Assume first that $|E| \leq n^{1/4-\nu}$.

\begin{figure}[!ht]
\begin{center}
\includegraphics[width=0.7\textwidth]{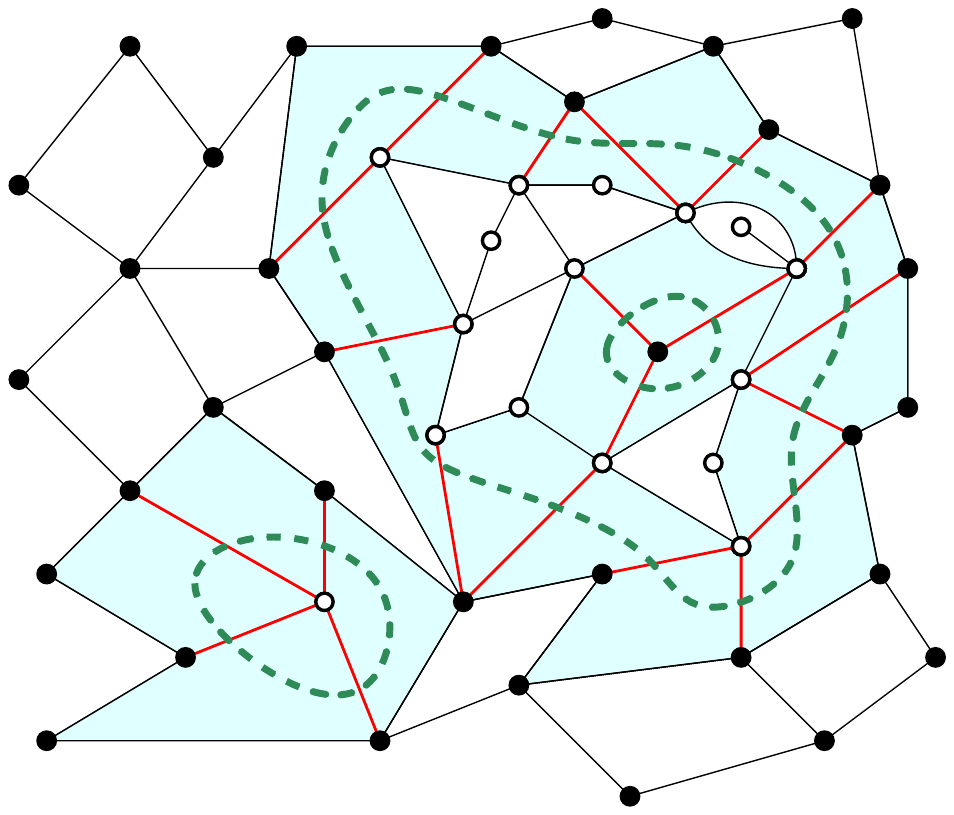}
\end{center}
\caption{The set $A$ is the set of all black vertices. Edges in $E(A,A^c)$, that is with one end in $A$ and the other in $A^c = V(Q) \setminus A$, are red. Faces in $S(A) \cap S(A^c)$ are colored in light blue, faces that are only in $S(A)$ are colored in grey, and faces that are only in $S(A^c)$ are colored in white. $S(A) \cap S(A^c)$ can be interpreted as the set of all faces crossed by the interfaces between the set of all white and the set of all black vertices (in dashed green).}
\label{Fig_Interface}
\end{figure}

Let $S(A)$ be the set of all faces that are incident to a vertex of $A$. We can see that $|\partial S(A)| \leq 2 |E| \leq 2n^{1/4-\nu}$ (observe that every face of $S(A)$ that is incident to at least one edge of $\partial S(A)$ is incident to at most two edges of $E$). 
It follows by \eqref{Eq_hyp_ineg_isop} that either $|S(A)| \leq 2^{4/3} |E|^{4/3} n^{2/3+\nu} \leq 3n^{1-\nu/3}$, or $|S(A)| > n/2$.

Let us show by contradiction that the second case is not possible whenever $\pi_Q(A) \leq 1/2$, where $\pi_{Q}$ is the stationary measure of the lazy random walk on $V(Q)$. If $|S(A)|>n/2$, then $|S(A)^c|<n/2$, and since $|\partial S(A)^c| = |\partial S(A)| \leq 2n^{1/4-\nu}$, \eqref{Eq_hyp_ineg_isop} applied to $S(A)^c$ gives $|S(A)^c| \leq 3n^{1-\nu/3}$. Consider now $S(A^c)$. It is immediate that $S(A) \cup S(A^c) = F(Q_n)$; on the other hand, any face of $S(A^c) \cap S(A)$ is incident to an edge of $E$, so $|S(A^c)\cap S(A)| \leq 2|E|$. Consequently, $|S(A^c)| = |S(A)^c| + |S(A^c)\cap S(A)| \leq 4 n^{1-\nu/3}$, and additionally $|A^c| \leq 4|S(A^c)| \leq 16 n^{1-\nu/3}$. By the expression of $\pi_Q$ and \eqref{Eq_max_degre_Q}, $\pi_Q(A^c) \leq \frac{|A^c|}{4n} \max_{x \in V(Q)} \deg_Q(x) < 1/2$. It follows that $\pi_Q(A) > 1/2$, which contradicts our assumption.

It follows that $|S(A)| \leq 2^{4/3} |E|^{4/3} n^{2/3+\nu}$ as soon as $\pi_Q(A) \leq 1/2$. The bounds $|A| \leq 4|S(A)|$ and $\pi_Q(A) \leq \frac{|A|}{4n} \max_{x \in V(Q)} \deg_Q(x)$ give that for every $A \subset V(Q)$ with $\pi_Q(A) \leq 1/2$ and $|E(A,A^c)| \leq n^{1/4-\nu}$,
\begin{equation*}
\pi_Q(A) \leq 2^{4/3} |E(A,A^c)|^{4/3} n^{-1/3+\nu} \ln n .
\end{equation*}
Considering separately every $A \subset V(Q)$ with $\pi_Q(A) \leq 1/2$ and $|E(A,A^c)| > n^{1/4-\nu}$, we have shown that
\begin{equation*}
\inf\l\{ \frac{|E(A,A^c)|^{4/3}}{4n\pi_{Q}(A)} \ : \ A \subset V(Q), \ \pi_Q(A) \leq 1/2 \r\} \geq n^{-2/3-2\nu} .
\end{equation*}
This holds as soon as \eqref{Eq_hyp_ineg_isop} and \eqref{Eq_max_degre_Q} are satisfied. As observed at the beginning of the proof, the probability that $Q_n$ satisfies both conditions goes to $1$ as $n \to \infty$, so Theorem \ref{Th_isoperimetric_inequality_V} follows. 

\end{proof}

\begin{proof}[Proof of Theorem \ref{Th_mixing_time}.] For $x,y \in V(Q)$, $x \neq y$, set $w_{x,y} \eqdef p_Q(x,y) \pi_Q(x) \geq 1/(4n)$. Note that for every $A \subset V(Q_n)$, $\sum_{x \in A, x \notin A} w_{x,y} = \frac{|E(A,A^c)|}{4n}$. 
By Theorem \ref{Th_isoperimetric_inequality_V}, with probability going to $1$ as $n\to\infty$, 
\begin{equation*}
\Phi \eqdef \inf\l\{ \frac{\sum_{x \in A, y\notin A} w_{x,y}}{\pi_{Q_n}(A)} \ : \ A \subset V(Q_n), \ \pi_{Q_n}(A) \leq 1/2 \r\} \geq n^{-3/4+\nu} .
\end{equation*}
We then apply \cite[Corollary 2.3]{jerrum1989approximating}: with probability going to $1$ as $n \to \infty$,
\begin{equation*}
\tau_{Q_n}(\ve) \leq n^{3/2+2\nu} (\ln(4n) + \ln(1/\ve) ) .
\end{equation*}

\end{proof}

\section[Standard hulls and density with the UIPQ]{Standard hulls in finite quadrangulations and density with the UIPQ}
\label{Sec_Standard hulls in finite quadrangulations and density with the UIPQ}

In this section, we establish Lemma \ref{Lemma_density_hulls_UIPQ_Qn}, which is stated below. Roughly speaking, this Lemma states that the probability of observing a given (simply connected) neighborhood of the root vertex in a uniform quadrangulation of finite size is smaller than the probability of observing the same neighborhood in the \UIPQ, times a factor that only involves the size of the neighborhood and the size of the finite quadrangulation. Lemma \ref{Lemma_density_hulls_UIPQ_Qn} will allow us to transfer a bound on the volume of hulls, established in Section \ref{Sec_tail_volume}, from the \UIPQ to finite quadrangulations in Section \ref{Section_Covering}.

Recall that a quadrangulation is a rooted planar map such that all its faces have degree $4$. We denote the set of all quadrangulations with $n$ faces by $\Q_n$, and the cardinality of $\Q_n$ by $\card{\Q_n}$. A quadrangulation $Q$ \emph{with a simple boundary} is a rooted planar map such that all its faces but (possibly) the face on the right of its root edge have degree $4$, and the boundary of this face is a simple cycle. The face on the right of its root edge is called the external face, and its boundary is called the boundary of $Q$; the other faces of $Q$ are called inner faces. We write $\Q_{n,p}$ for the set of all rooted quadrangulations with a simple boundary of length $2p$ and $n$ inner faces. Note that $\Q_{n,p} = \emptyset$ if $p > n+1$ as an easy consequence of Euler's formula. By convention, we fix that $\Q_0$ contains one quadrangulation, the “edge-quadrangulation”, and $\Q_{0,1}$ contains the unique map with one face and one edge. 
By \cite[Section 6.2]{CurienLeGall2015scaling}, for every $n\geq 0$ and $1 \leq p\leq n+1$,
\begin{align}
\label{Eq_formula_Qn}
\card{\Q_n} &= 3^{n} \frac{2 (2n)!}{n!(n+2)!} , \\ 
\label{Eq_formula_Qnp}
\card{\Q_{n,p}} &= 3^{n-p} \frac{(3p)!(2n+p-1)!}{(n+1-p)!p!(2p-1)!(n+2p)!} .
\end{align}
The following asymptotics come from \eqref{Eq_formula_Qn} and \eqref{Eq_formula_Qnp}:
\begin{equation}
\label{Eq_asymp_Qn}
\card{\Q_{n}} \usim n \infty \frac{2}{\sqrt \pi} n^{-5/2} 12^n 
\end{equation}
and for every $p \geq 1$,
\begin{equation}
\card{\Q_{n,p}} \usim n \infty C_p n^{-5/2} 12^n 
\end{equation}
with
\begin{equation}
C_p = 2^p 3^{-p} \frac{(3p)!}{p!(2p-1)!} \frac{1}{2\sqrt \pi} .
\end{equation}

\begin{lemma}
\label{Lemma_asymptotic_cardinals}
There exists a constant $c>0$ such that for every $n>0, 1\leq p\leq n+1$,
\begin{equation}
\label{Eq_majoration_cardinaux}
\card{\Q_{n,p}} \leq c \cdot C_p n^{-5/2} 12^n .
\end{equation}
\end{lemma}

\begin{proof}
If $p = n+1$, then \eqref{Eq_majoration_cardinaux} is directly verified. 
Using Stirling's formula, we can find a positive constant $c$ such that for every $n > 0$, for every $1 \leq p \leq n$, 
\begin{align}
\card{\Q_{n,p}} &= 3^{n-p} \frac{(3p)!}{p!(2p-1)!} \frac{(2n+p-1)!}{(n+1-p)!(n+2p)!} \nonumber \\
&\leq c \cdot 2\sqrt{\pi} C_p 2^{-p} 3^{n} \frac{(2n+p+1)^{2n+p+1}}{(n+1-p)^{n+1-p}(n+2p)^{n+2p}} \frac{1}{(2n+p+1)(2n+p)} \\
&\qquad \qquad \cdot \sqrt{\frac{2\pi (2n+p+1)}{4\pi^2 (n+1-p)(n+2p)}} \nonumber \\
&= c \cdot C_p n^{-5/2} 12^{n} \frac{(1+\frac{p+1}{2n})^{2n+p+1/2}}{(1-\frac{p-1}{n})^{n-p+3/2}(1+\frac{2p}{n})^{n+2p+1/2}} \frac{1}{1+\frac{p}{2n}} \nonumber \\
\label{Eq_asymp_Qnp_proof}
&= c \cdot C_p n^{-5/2} 12^{n} \exp\l( n f_n\pfrac{p-1}{n} \r)  \frac{1}{1+\frac{p}{2n}} ,
\end{align}
where for every $n>0$ and $x \in [0,1)$,
\begin{multline*}
f_n(x) \eqdef \l(2+x+\frac{3}{2n} \r) \ln\l(1+\frac{x}{2} + \frac{1}{n} \r)  -
\l(1-x+\frac{1}{2n}\r)  \ln\l(1-x \r) \\ -
\l(1+2x+\frac{5}{2n}\r) \ln \l( 1+2x+\frac{2}{n} \r) .
\end{multline*}
Let $\chi_n(u) = \l(u + \frac{1}{2n} \r) \ln u$, and rewrite 
\begin{equation*}
f_n(x) = 2 \chi_n\l( \frac{a+b}{2} \r) - \chi_n(a) - \chi_n(b) - \frac{1}{2n} \ln \l(1+\frac{x}{2} + \frac{1}{n} \r) 
\end{equation*}
for $a = 1-x$ and $b= 1+2x+2/n$. Since $\frac{1}{2n} \ln \l(1+\frac{x}{2} + \frac{1}{n} \r) \geq 0$, and $\chi_n$ is convex on $[1/(2n), \infty)$, taking $x = (p-1)/n$ (ensuring $a \geq 1/n$) gives $f_n\pfrac{p-1}{n} \leq 0$.

\end{proof}

For every quadrangulation $Q$ and $v \in V(Q)$, we write $B_Q(v,r)$ for the $r$-ball (or ball of radius $r$) centered at $v$, defined as the union of faces that are incident to a vertex at distance at most $r-1$ from $v$. We define the standard $r$-hull $\Hull_Q(v,r)$ as the union of the ball with the connected components of its complement, excluding the one containing the most faces (if there is an ambiguity, we lift it by a deterministic rule). We view $\Hull_Q(v,r)$ as a quadrangulation with a simple boundary (its external face corresponds to the excluded component of the complement of $B_Q(v,r)$). 
We say a map $b$ is an \emph{admissible standard $r$-hull} if there exists a quadrangulation $Q$ and $v \in V(Q)$ such that $b = \Hull_Q(v,r)$.

Let $Q_n$ be uniformly distributed over $\Q_n$, and $Q_\infty$ be the \UIPQ. We denote the root vertex of $Q_n$, \resp $Q_\infty$ by $\rho_n$, \resp $\rho_\infty$.

\begin{lemma}
\label{Lemma_density_hulls_UIPQ_Qn}
There exists $c'>0$ such that for every $r>0$, for every $n>0$, for every admissible standard $r$-hull $b$ with $N$ inner faces, $0<N<n$,
\begin{equation}
\frac{\P(\Hull_{Q_n}(\rho_n,r) = b)}{\P(\Hull_{Q_\infty}(\rho_\infty,r) = b)} \leq c' \pfrac{n-N}{n}^{-5/2} .
\end{equation}
\end{lemma}

\begin{proof}
Let $b$ be a rooted planar map with a distinguished face, such that all faces but the distinguished one have degree 4. We assume that $b$ has $N$ inner faces (not counting the distinguished face), and that its distinguished face has simple boundary of length $2p$. Finally, we mark an oriented edge on the boundary of the distinguished face of $b$ by some deterministic procedure that only involves $b$. 

Let $Q$ be a quadrangulation with $n$ faces and root vertex $\rho$. Suppose that $n>N$. Then $b \subset Q$ holds if and only if $Q$ is obtained by gluing a quadrangulation $Q'$ with simple boundary of length $2p$ and $n-N$ inner faces inside the distinguished face of $b$, so that the root edge of $Q'$ is glued on the marked edge of $b$. From \eqref{Eq_asymp_Qn}, we can find $c''$ large enough such that for every $n>0$, we have $(\card{\Q_n})^{-1} \leq c'' n^{5/2} 12^{-n}$. It follows that for every $n>N$, using \lemref{Lemma_asymptotic_cardinals}, 
\begin{align*}
\P(b \subset Q_n) &= \frac{\card{\Q_{n-N,p}}}{\card{\Q_n}} \\
&\leq \frac{c\, c'' \cdot C_p (n-N)^{-5/2} 12^{n-N}}{n^{-5/2} 12^n} \\
&\leq c\, c'' \cdot C_p \pfrac{n-N}{n}^{-5/2} 12^{-N} .
\end{align*}
On the other hand, since the \UIPQ is the local limit of $Q_n$ as $n$ goes to infinity, 
\begin{align*}
\P(b \subset Q_\infty ) &= \lim_{n \to \infty} \P(b \subset Q_n) \\
&= \lim_{n\to\infty} \frac{\card{\Q_{N-n,p}}}{\card{\Q_n}} \\
&= \lim_{n\to\infty} \frac{C_p (n-N)^{-5/2} 12^{n-N}}{\frac{2}{\sqrt \pi} n^{-5/2} 12^n} \\
&= \frac{\sqrt \pi}{2} C_p 12^{-N}.
\end{align*}
If $b$ is an admissible standard $r$-hull, then 
$\P(\Hull_{Q_n}(\rho_n,r) = b) \leq \P(b \subset Q_n)$, and $\P(\Hull_{Q_\infty}(\rho_\infty,r) = b) = \P(b \subset Q_\infty)$. 
Combining the two bounds yields the lemma.
\end{proof}

\section[Laplace transform of the volume of hulls]{Laplace transform and tail estimates of the volume of hulls in the UIPQ}
\label{Sec_Laplace_transform_hull_volume}

The goal of this section is to establish a bound on the probability that the volume of a standard hull of the \UIPQ is large. We establish such a bound for truncated hulls, for which the analysis is simpler thanks to the skeleton decomposition of the \UIPQ. The bound is then transferred to standard hulls using the inclusion \eqref{Eq_inclusion_hulls}.

\subsection{Preliminaries}

Following the definition in \cite[Section 2]{gall2017separating}, a truncated quadrangulation is a rooted planar map with a distinguished face, called the external face, such that
\begin{romenum}
\item the external face $f$ has simple boundary,
\item every edge incident to $f$ is also incident to a triangular face, and these triangular faces are all distinct, 
\item every other face has degree 4.
\end{romenum}
The faces that are not the external face are called inner faces. 
Consider a one-ended infinite quadrangulation $Q$ of the plane and assume that it is drawn on the plane in such a way that any compact set of the plane intersects only a finite number of faces of $Q$ (the \UIPQ can be drawn in this way). 
Label every vertex of $Q$ by its distance to the root vertex. Fix $r>0$, and in every face whose incident vertices have label (in clockwise order) $r$, $r-1$, $r$, $r+1$, draw an edge between the two corners labeled $r$. The collection of added edges forms a union of (not necessarily disjoint) simple cycles, one of which is ``maximal'' in the sense that the connected component of its complement containing the root vertex of $Q$ also contains every other added edge (see \cite[Lemma 5]{gall2017separating}). We denote this maximal cycle by $\partial_r Q$. 
Adding the edges of $\partial_r Q$ to $Q$ and removing the infinite connected component of the complement of $\partial_r Q$ gives a truncated quadrangulation, which we call the \emph{truncated $r$-hull} centered at the root vertex of $Q$, or in short ``the truncated $r$-hull of $Q$'', and denote by $\Htr_{Q}(r)$.

Let us explain why \eqref{Eq_inclusion_hulls} holds. 
\cite[Lemma 5]{gall2017separating} ensures that every vertex of $Q$ at distance less than or equal to $r+1$ is contained in $\Htr_{Q}(r+1)$. Since the standard $r$-hull of $Q$ is bounded by a cycle of edges of $Q$ that only visits vertices at distance $r$ or $r+1$ of the root vertex, it follows that $\Hull_{Q}(r) \subset \Htr_Q(r+1)$. 
On the other hand, $\partial_r Q$ is entirely contained in faces of $B_Q(r)$, so when adding the finite connected components of its complement one gets $\Htr_Q(r) \subset \Hull_Q(r)$.

We now describe the \emph{skeleton decomposition} of $Q$, which encodes the structure of every $\Htr_Q(r)$ using a forest of plane trees and a collection of truncated quadrangulations. This decomposition was first described in \cite{Krikun2008local}; see \cite[Section 2.3]{gall2017separating} or \cite[Section 2]{lehe2019fpp} for a more detailed explanation that is compatible with our notations. 
 Let $Q'$ be the following modification of $Q$: split the root edge of $Q$ into a face of degree two, add a loop inside this face that is incident to the root vertex of $Q$, and root $Q'$ at the added edge so that the face of degree 1 lies on the right of the new root edge, see \figref{Fig_Root}. Let $\partial_0 Q$ be the cycle of $Q'$ made of the single root edge of $Q'$. 
 For every $r>0$ and every edge $e$ of $\partial_r Q$, $e$ splits a face $f$ of $Q'$ into two triangular faces; the one that is contained in $\Htr_Q(r)$ is called the \emph{downward triangle with top edge $e$}. The downward triangles cut the part of $Q'$ outside $\partial_{0} Q$ into a collection of \emph{slots}, filled with finite maps (possibly reduced to a single edge), see \figref{Fig_Downward_Triangles}. Every slot contained between $\partial_{r-1} Q$ and $\partial_r Q$ is incident to a unique vertex $v$ of $\partial_r Q$: we say that the slot is associated with the edge of $\partial_r Q$ with tail vertex $v$, where the edges of $\partial_r Q$ are oriented clockwise.

\begin{figure}[!ht]
\begin{center}
\includegraphics[scale=1.5]{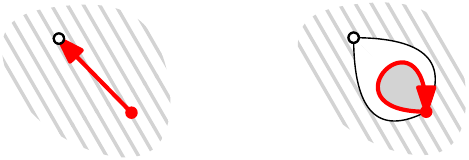}
\end{center}
\caption{We split the root edge of $Q$ (left, in red), and add an edge in the manner illustrated here to obtain $Q'$. $Q'$ is rooted at the new edge. The external face of $Q'$ is in grey.}
\label{Fig_Root}
\end{figure}

\begin{figure}[!ht]
\begin{center}
\includegraphics[width=0.7\textwidth]{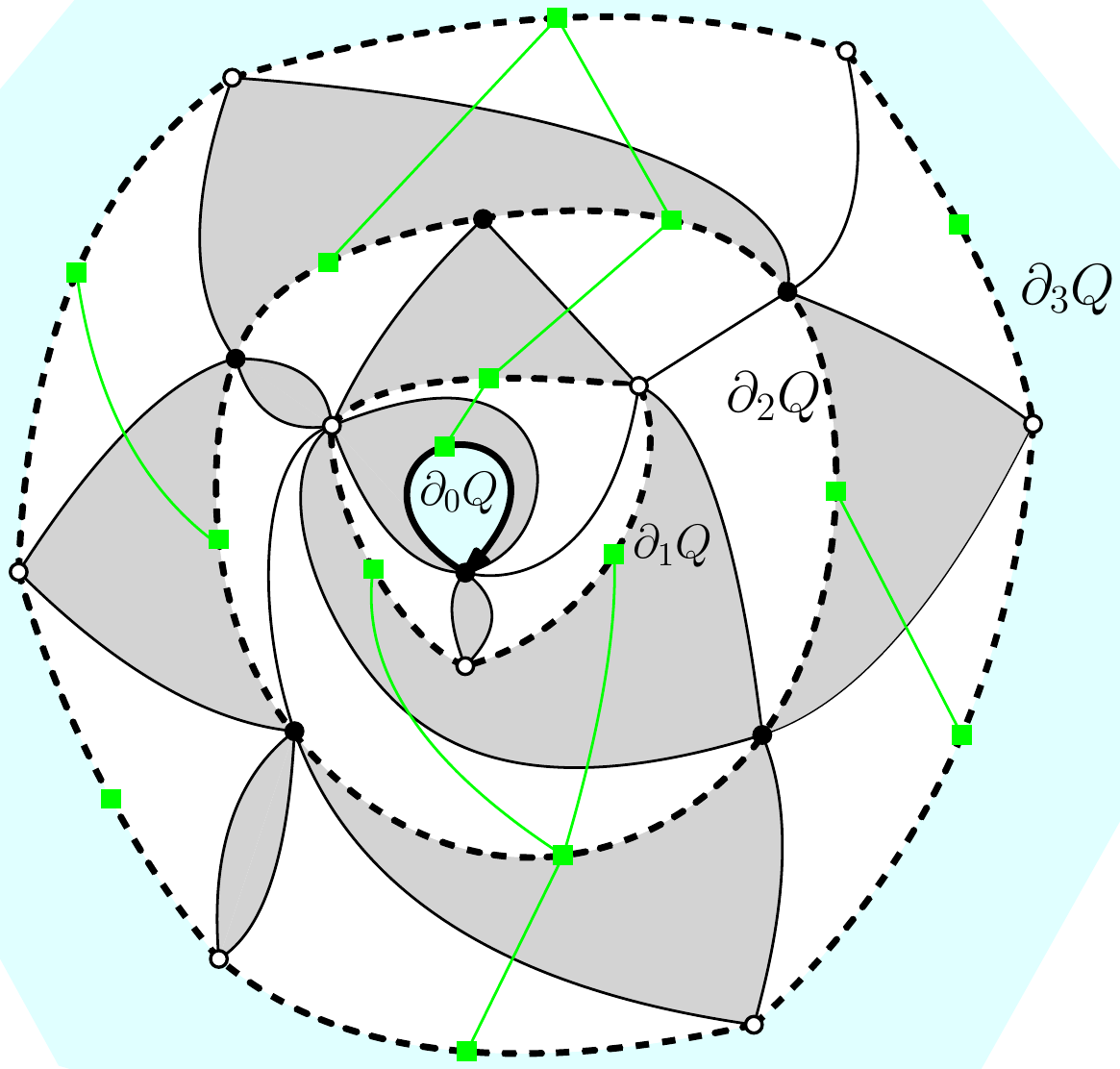}
\end{center}
\caption{The part of some quadrangulation $Q'$ contained between $\partial_0 Q$ and $\partial_{3} Q$. Downward triangles are in white, slots in grey, edges of $\partial_s Q$ in dashed thick lines (for $s\geq 1$, when they are not edges of $Q'$) or in thick lines (for s=0). The outside of this part of the map is in cyan. Also included is the genealogical relation, represented by the green forest (rooted at edges of $\partial_{3} Q$).} 
\label{Fig_Downward_Triangles}
\end{figure}

We define the following genealogical relation on edges of $\cup_{r\geq 0} \partial_r Q$: for every $r>0$ and every edge $e$ of $\partial_r Q$, $e$ is the parent of every edge of $\partial_{r-1} Q$ that is incident to the slot associated to $e$. 
The unique edge of $\partial_0 Q$ has no child. For every $r>0$, write $\FF_r$ for the collection of all edges in $\partial_s Q$, $0 \leq s \leq r$ together with its genealogical relation, seen as a planar forest, and number its trees from $1$ to $|\partial_r Q|$ according to the clockwise order on their roots in such a way that the tree with index $1$ contains the unique edge of $\partial_0 Q$. Then $\FF_r \in \F^{o}_{r,q}$ the set of all $(r,q)$-admissible forests, where we say (slightly modifying \cite[Section 2.3]{gall2017separating}) that a forest is $(r,q)$-admissible if
\begin{romenum}
\item it consists of $q$ rooted plane trees,
\item there is exactly one vertex at generation $r$, and no vertex at generation $r+1$ or more,
\item the vertex at generation $r$ is contained in the first tree.

\end{romenum}

Let $r>0$, and let $e$ be an edge of $\partial_r Q$. Consider the map $M_e$ filling the slot associated to $e$, and let $v$ be the vertex of $M_e$ that belongs to $\partial_r Q$ (it is the tail of $e$). The following modification changes $M_e$ into a truncated quadrangulation: add an edge inside the unbounded face of $M_e$ to create a triangular face, in such a way that $v$ is not anymore incident to the unbounded face, and root the $M_e$ at the added edge in such a way that the unbounded face lies on its right. See \figref{Fig_Slots}. We note that, if $c_e$ is the number of offspring of the edge $e$, then $M_e$ has perimeter $c_e+1$.

\begin{figure}[!ht]
\begin{center}
\includegraphics[width=1.0\textwidth]{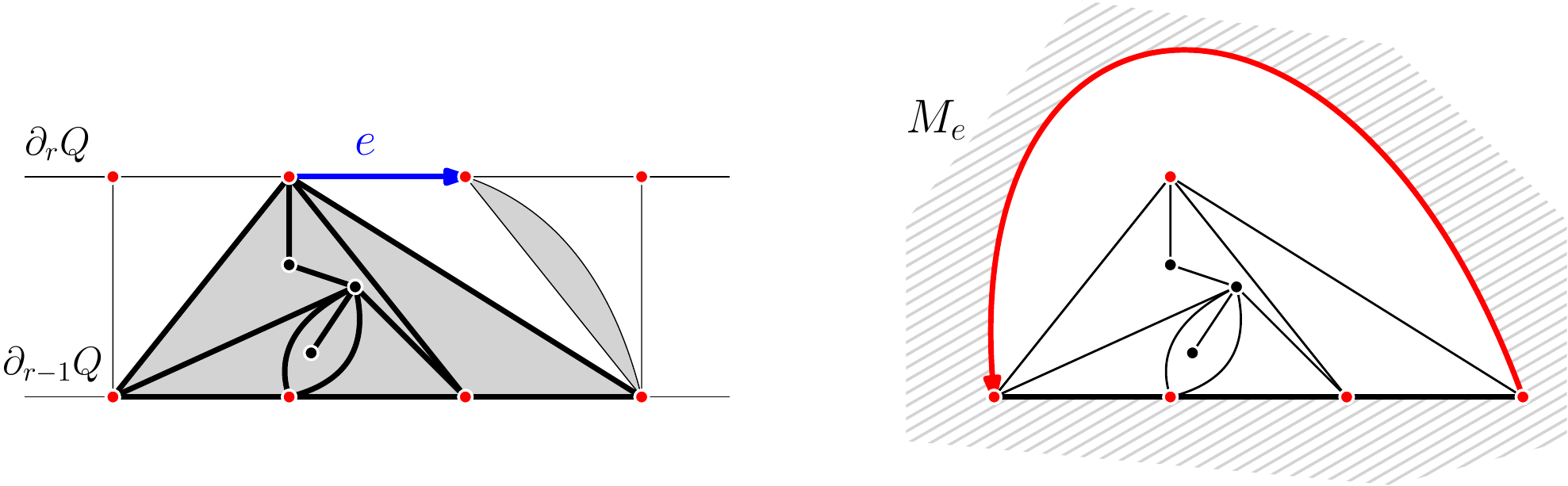}
\end{center}
\caption{The content of the slot associated to some edge $e$ of $\partial_r Q$ (left, thick black lines) can be seen as a truncated quadrangulation $M_e$ after adding one edge above the tail of $e$, and rooting $M_e$ at that new edge (right).}
\label{Fig_Slots}
\end{figure}

An admissible truncated $r$-hull $H$ is a truncated quadrangulation such that there exists an infinite quadrangulation $Q$ with $H = \Htr_Q(r)$. Consider such an infinite quadrangulation $Q$, and let $\FF^*_r$ be the collection of all vertices of $\FF_r$ at generation at most $r-1$. 
The data of the ``skeleton'' $\FF_r$ and the truncated quadrangulations filling the slots $(M_e)_{e \in \FF^*_r}$ is a function of $H$; it is called the \emph{skeleton decomposition of $H$}. Conversely, given a forest $\FF \in \F^{o}_{q,r}$ and a collection $(M_e)_{e \in \FF^*}$ of truncated quadrangulations such that $M_e$ has perimeter $c_e$ for every $e\in\FF^*$, we can recover a unique admissible truncated $r$-hull: the skeleton decomposition is bijective.

\subsection{Computing the Laplace transform of the hull volume}
\label{Sec_computing_Laplace}

The following Theorem draws inspiration from \cite{menard2018volumes}, which proves a similar result in the uniform infinite plane triangulation. In fact, Ménard also establishes the law of the volume of hulls conditioned on their perimeter. Since we are only concerned with the volume of truncated hulls, for concision we only prove the equivalent of their Theorem 1. We are nevertheless confident that more results of \cite{menard2018volumes} could be generalized to the \UIPQ. 

If $M$ is a truncated quadrangulation, we denote the number of inner faces of $M$ by $|M|$. 

\begin{theorem}[Laplace transform of the volume of truncated hulls]
\label{Th_Laplace}
For every $r>0$, $p \in (0,1)$,
\begin{equation}
\label{Eq_Laplace_transform_of_truncated_hull_volume}
\E\l[ (1-p^2)^{|\Htr_{Q_\infty}(r)|} \r] = (1-p)^{-1} \l(\KK \circ \psi_r\r)'(0) ,
\end{equation}
where 
\begin{align}
\nonumber \KK(t) &= \frac{3}{4}\sqrt{\frac{8+t}{t}} ,\\
\label{Eq_expression_psi}
\psi_r(u) &= p+ \frac{6p}{-1+\sqrt{\frac{2(1-p)}{p+2}} \cosh(B(u)+ry)} ,\\
\nonumber y &= \arccosh\pfrac{2p+1}{1-p} ,\\
\nonumber B(u) &= \arccosh\l( \sqrt{\frac{p+2}{2(1-p)}}\l[1+\frac{6p}{(1-u)(1-p)} \r] \r) .
\end{align}
\end{theorem}

\begin{proof}

Let us briefly recall some results about the enumeration of truncated quadrangulations. 
Denote the set of all truncated quadrangulations with $n$ inner faces and boundary length $p$ by $\Qtr_{n,p}$. 
\cite[Section 2.2]{Krikun2008local} gives an explicit expression for the generating function of truncated quadrangulations:
\begin{align}
\label{Eq_expression_U}
\nonumber U(x,y) &\eqdef \sum_{n >0, p>0} \card{\Qtr_{n,p}} x^n y^p \\
&= \frac{1}{2}\l( y-xy^2-1+\sqrt{y^2-2xy^3-2y+4xyq(x)+(xy^2-1)^2} \r) ,
\end{align}
with $q$ the generating function of quadrangulations:
\begin{equation*}
q(x) \eqdef \sum_{n \geq 0} \card{\Q_n} x^n =  \frac{4}{3} \frac{2\sqrt{1-12x}+1}{(\sqrt{1-12x}+1)^2} .
\end{equation*}
Singularity analysis gives the asymptotics of the number of truncated quadrangulations:
\begin{equation*}
\card{\Qtr_{n,p}} \usim n \infty \vk_p n^{-5/2} 12^n ,
\end{equation*}
with 
\begin{equation}
\label{Eq_expression_K}
\sum_{p\geq 1} \vk_p y^p = \frac{128 \sqrt 3}{\sqrt \pi} \frac{y}{\sqrt{(18-y)(2-y)^3}} 
\end{equation}
and
\begin{equation}
\label{Eq_expression_K1}
\vk_1 = \frac{32}{\sqrt{3\pi}} .
\end{equation}

We now proceed with the proof. 
Let $H$ be an admissible truncated $r$-hull with perimeter $q$, let $\FF$ be its skeleton. 
For every $e \in \FF^*$, let $c_e$ be the number of offspring of $e$ and let $M_e$ be the truncated quadrangulation filling the slot associated to $e$. 
From the proof of \cite[Lemma 6]{gall2017separating}:
\begin{equation}
\label{Eq_relation_offspring_numbers}
\sum_{e \in \FF^*} (c_e-1) = 1-q ,
\end{equation}
\begin{equation}
\label{Eq_relation_volume_skeldecomp}
|H| = q-1+\sum_{e \in \FF^*} (|M_e|-1).
\end{equation}
The convergence of finite quadrangulations towards the \UIPQ \cite[Theorem 1]{Krikun2008local}, together with \cite[(7)]{gall2017separating}, ensures that 
\begin{equation}
\label{Eq_brutal_law_tr_hull}
\P(\Htr_{Q_\infty}(r) = H) = \frac{\vk_q}{\vk_1} 12^{q-|H|-1} = \frac{\vk_q}{\vk_1} \prod_{e \in \FF^*} 12^{-|M_e|+1} .
\end{equation}
Let us introduce two variables $s$ and $t$. We will later choose their value in a suitable way. 
For every $s,t>0$, multiply \eqref{Eq_brutal_law_tr_hull} by $t^{q-1+\sum_{e \in \FF^*} (c_e-1)} s^{q-1+\sum_{e \in \FF^*} (|M_e|-1)} = s^{|H|}$ (by \eqref{Eq_relation_offspring_numbers} and \eqref{Eq_relation_volume_skeldecomp}):
\begin{align*}
s^{|H|}\P(\Htr_{Q_\infty}(r) = H) &=\frac{(st)^q \vk_q}{(st)^1 \vk_1} \prod_{e \in \FF^*} t^{c_e-1} \pfrac{s}{12}^{|M_e|-1} .
\end{align*}
Summing over all admissible truncated $r$-hulls, we get for every $s,t > 0$:
\begin{align}
\nonumber \E\l[ s^{|\Htr_{Q_\infty}(r)|} \r] &= \sum_{H} s^{|H|} \P(\Htr_{Q_\infty}(r) = H) \\
\nonumber &= \sum_{q>0} \frac{(st)^q \vk_q}{(st)^1 \vk_1} \sum_{\FF \in \F^o_{r,q}} \prod_{e \in \FF^*} t^{c_e-1} \sum_{|\partial M_e| = c_e+1} \pfrac{s}{12}^{|M_e|-1} \\
\label{Eq_Laplace_1}
&= \sum_{q>0} \frac{(st)^q \vk_q}{(st)^1 \vk_1} \sum_{\FF \in \F^o_{r,q}} \prod_{e \in \FF^*} \vT(c_e) ,
\end{align}
where $\vT(c) \eqdef t^{c-1} \sum_{|\partial M| = c+1} \pfrac{s}{12}^{|M|-1}$, where the sum in the definition of $\vT$ is over the set of all truncated quadrangulations with perimeter $c+1$. 
We recognize the generating series of truncated quadrangulations: 
\begin{equation*}
\sum_{|\partial M| = c+1} \pfrac{s}{12}^{|M|-1} = \frac{12}{s} [y^{c+1}] U\l(\frac{s}{12},y\r) .
\end{equation*}
The generating function of $\vT$ is thus 
\begin{equation*}
\Phi(u) \eqdef \sum_{c \geq 0} \vT(c) u^c = \frac{12}{st} \sum_{c \geq 0} (t u)^{c} [y^{c+1}] U\l(\frac{s}{12},y\r) = \frac{12}{s t^2 u} \l(U\l(\frac{s}{12}, t u\r) - U\l(\frac{s}{12},0\r) \r).
\end{equation*}
We claim that if $s = 1-p^2$ and $t=2/(1+p)$ for some $p \in [0,1)$, then $\vT$ defines a probability distribution on nonnegative integers. Indeed, using \eqref{Eq_expression_U} we can rewrite $\Phi$ as follows:
\begin{multline}
\label{Eq_phi}
\Phi(u) = \frac{1}{2u(1-p)}\Big( u^2 p-u^2-3p+6u-3 \\
+ (1-u)\sqrt{p^2u^2+2p^2u-2pu^2+9p^2+8pu+u^2+18p-10u+9} \Big) .
\end{multline}
Our claim is equivalent to $\Phi(1)=1$, which is immediate.

Let $\F_{r,q}$ be the set of all planar forests satisfying conditions (i) and (ii) of the definition of admissible forests. Every forest $\FF \in \F^o_{r,q}$ gives rise to $q$ different forests of $\F_{r,q}$ by applying one of the $q$ circular permutations of trees of $\FF$. 
For every $\FF_\bullet \in \F_{r,q}$, the probability that a \BGW forest with offspring distribution $\vT$ is equal to $\FF_\bullet$ is $\prod_{e \in \FF_\bullet^*} \vT(c_e)$. It follows that 
$$q\sum_{\FF \in \F^o_{r,q}} \prod_{e \in \FF^*} \vT(c_e) = \sum_{\FF_\bullet \in \F_{r,q}} \prod_{e \in \FF_\bullet^*} \vT(c_e) = \P(Y_r = 1) = [u] \l( \Phi^{(r)}(u) \r)^q , $$
where $(Y_r)_{r \geq 0}$ is a \BGW process with offspring distribution $\vT$ such that $Y_0 = q$ and $\Phi^{(r)}$ denotes the $r$-th iterate of $\Phi$.  
Let us rewrite \eqref{Eq_Laplace_1} with our special choice of $s$ and $t$:
\begin{align*}
\E\l[ (1-p^2)^{|\Htr_{Q_\infty}(r)|} \r] &= \frac{1}{2(1-p) \vk_1} \sum_{q \geq 1} \frac{1}{q} \sum_{\FF_\bullet \in \F_{r,q}} (2(1-p))^q \vk_q \prod_{e \in \FF_\bullet^*} \vT(c_e) \\
&= \frac{1}{2(1-p) \vk_1} \sum_{q \geq 1} \frac{1}{q} (2(1-p))^q \vk_q [u] \l( \Phi^{(r)}(u) \r)^q \\
&= \frac{1}{\vk_1} \l(\Phi^{(r)}\r)'(0) \sum_{q \geq 1} \vk_q \l(2(1-p) \Phi^{(r)}(0) \r)^{q-1} \\
&= \l(\Phi^{(r)}\r)'(0) K(2(1-p) \Phi^{(r)}(0)) ,
\end{align*}
with 
$$K(x) = \frac{1}{\vk_1} \sum_{k \geq 1} \vk_k x^{k-1} = 12 \cdot (18-x)^{-1/2} (2-x)^{-3/2} $$
by \eqref{Eq_expression_K} and \eqref{Eq_expression_K1}. 
Let $\KK(t) = \frac{3}{4}\sqrt{\frac{8+t}{t}}$, then $K(x) = (2\KK(1-x/2))' = - \KK'(1-x/2)$, thus
\begin{align}
\E\l[ (1-p^2)^{|\Htr_{Q_\infty}(r)|} \r] &= - \l(\Phi^{(r)}\r)'(0) \KK'( 1-(1-p) \Phi^{(r)}(0)) \nonumber\\
\label{Eq_Laplace_2}
&= (1-p)^{-1} \psi_r'(0) \KK'(\psi_r(0))  \\
\nonumber &= (1-p)^{-1} \l(\KK \circ \psi_r\r)'(0) ,
\end{align}
where $\psi_r(u) = 1-(1-p)\Phi^{(r)}(u)$. 
The closed formula for $\psi_r$ is provided by the next Lemma. This finishes the proof of the Theorem.
\end{proof}

\begin{lemma}
\label{Lemma_computation_of_phi}
Let $\Phi^{(r)} = \Phi \circ ... \circ \Phi$ be the $r^\text{th}$ iterate of $\Phi$. 
The following holds for every $r \geq 0$, for $u\in[0,1)$ and $p \in (0,1)$:
\begin{equation*}
\Phi^{(r)}(u) = 1 - \frac{6p}{(1-p)\l( -1+\sqrt{\frac{2(1-p)}{p+2}} \cosh(B(u)+ry) \r)}
\end{equation*}
with $y$ and $B(u)$ as in \thref{Th_Laplace}. Furthermore, \eqref{Eq_expression_psi} holds.
\end{lemma}

\begin{proof}

Let $T_r(u)= \frac{2}{1-\Phi^{(r)}(u)}$. 
Writing $g(t) \eqdef \frac{2}{1-\Phi\l( 1-2/t \r)}$, we check that for every integer $r\geq 0$, $T_{r+1}(u) = g(T_r(u))$. Let us compute a more explicit expression for $g$: using \eqref{Eq_phi}, 
\begin{multline*}
\Phi(1-2/t) = \frac{1}{t(t-2)(1-p)}\Big( t^2 -pt^2 -2pt +2p -4t -2
\\
+ 2\sqrt{3p(p+2)t^2 + 2(1-p)(p+2)t + (1-p)^2} \Big) .
\end{multline*}
Then
\begin{align*}
g(t) &= \frac{2t(t-2)(1-p)\l( 2t+4pt-2p+2 + 2\sqrt{3p(p+2)t^2 + 2(1-p)(p+2)t + (1-p)^2} \r)}{4(1-p)^2 t (t-2)} \\
&= \frac{(2p+1)t+(1-p)+ \sqrt{3p(p+2)t^2 + 2(1-p)(p+2)t + (1-p)^2}}{1-p} .
\end{align*}
In summary, for every $r\geq 0$
\begin{equation}
\label{Eq_recurrence_initiale_sur_T}
T_{r+1}(u) = \frac{(2p+1)T_r(u)+(1-p)+\sqrt{3p(p+2)T_r(u)^2 + 2(1-p)(p+2)T_r(u) + (1-p)^2}}{1-p} .
\end{equation}
We prove by induction that for every $r\geq 0$, for every $p \in (0,1)$ and $u \in [0,1)$,
\begin{equation}
\label{Eq_rec_T}
T_r(u) = \frac{1-p}{3p}\l( -1+\sqrt{\frac{2(1-p)}{p+2}} \cosh(B(u)+ry) \r) .
\end{equation}
For $r=0$, 
\begin{align*}
\frac{1-p}{3p}\l( -1+\sqrt{\frac{2(1-p)}{p+2}} \cosh(B(u)) \r) &= \frac{1-p}{3p}\l( -1 + 1 + \frac{6p}{(1-u)(1-p)} \r) \\
&= \frac{2}{1-u} \\
&= T_0(u) .
\end{align*}
Let $r\geq 0$, and assume that \eqref{Eq_rec_T} holds. Substituting in \eqref{Eq_recurrence_initiale_sur_T} gives
\begin{equation}
\label{Eq_rerec sur T}
T_{r+1}(u) = \cosh(y)T_r(u) + 1 + \sqrt{1+\frac{2(p+2)}{1-p}T_r(u) + \sinh(y)^2 T_r(u)^2 }  
\end{equation}
One checks that 
\begin{align*}
\sqrt{1+\frac{2(p+2)}{1-p}T_r(u) + \sinh(y)^2 T_r(u)^2} &=  \frac{1-p}{3p} \sqrt{\frac{2(1-p)}{p+2}} \sinh(y) \sinh(B(u)+ry) , \\
\cosh(y) T_r(u) +1 &= \frac{1-p}{3p} \sqrt{\frac{2(1-p)}{p+2}} \cosh(y) \cosh(B(u)+ry) - \frac{1-p}{3p} .
\end{align*}
Substituting in \eqref{Eq_rerec sur T} gives
\begin{align*}
T_{r+1}(u) &= -\frac{1-p}{3p} +  \frac{1-p}{3p} \sqrt{\frac{2(1-p)}{p+2}} \cosh(B(u)+(r+1)y). 
\end{align*}
This proves \eqref{Eq_rec_T}. A straightforward rewriting then gives the expressions in \lemref{Lemma_computation_of_phi}. 

\end{proof}

\subsection{Tail estimates for the volume of truncated hulls}
\label{Sec_tail_volume}

The goal of this section is to get a Taylor expansion of $\E\l[ e^{-\vl|\Htr_{Q_\infty}(r)|} \r]$, with a remainder term that we control simultaneously for every $r \geq 0$.

\begin{lemma}
\label{Lemma_dl_Laplace}
There exists $\vl_0,c >0$ such that for every $r>0$, for every $0\leq \vl \leq \vl_0$:
\begin{multline}
\label{Eq_dl_of_Laplace}
\Bigg| \E\l[ e^{-\vl |\Htr_{Q_\infty}(r)|/r^4} \r] - \Bigg( 1- \frac{(r+3)(6r^4+36r^3+87r^2+99r+44)}{4r^3(2r+3)^2} \vl \\
 + \frac{(r+3)(r+1)^3(r+2)^3}{2r^5(2r+3)^2} \vl^{3/2} \Bigg) \Bigg| \leq c \vl^2 .
\end{multline}
\end{lemma}

\begin{proof}
The Proposition is proved by computing Taylor expansions of $\psi_r'(0)$ and of $\psi_r(0)$ near $p=0$, and then using \eqref{Eq_Laplace_2}. 
For readability we omitted the detailed computations of the coefficients; a maple sheet is provided in Annex. 
By a careful development in $p$ of each individual term in $T_r(0)$, one can find $C_1>0$ and $x_1>0$ (that do not depend on $r$) such that for every $r>0$ and $p = xr^{-2}$ with $0< x < x_1$,
\begin{multline*}
\Bigg|\frac{r}{(r+1)(r+2)(r+3)} T_r(0) \\ - \Bigg( \frac{r}{r+3} + \frac{1}{2}x + \frac{r^2+3r+1}{10r^2} x^2 + \frac{3r^4+18r^3+41r^2+42r+36}{280r^4} x^3 \Bigg) \Bigg| \leq C_1 x^4 ,
\end{multline*}
Given that $\psi_r(u) = 1-(1-p)\l(1-\frac{2}{T_r(u)}\r) = p + \frac{2(1-p)}{T_r(u)}$, similar expansions hold for $\psi_r(0)$, as well as for $T_r'(0)$ and $\psi'_r(0)$. We then compute the expansion of 
\begin{equation*}
\E\l[ (1-p^2)^{|\Htr_{Q_\infty}(r)|} \r] = -\frac{3}{1-p} \frac{\psi_r' (0)}{(8+\psi_r(0))^{1/2} \psi_r (0)^{3/2}} 
\end{equation*}
given by \eqref{Eq_Laplace_2}: there exists $x_2, C_2>0$ such that for every $r>0$ and $0<x<x_2$,
\begin{multline*}
\Bigg| \E\l[ (1-x^2 r^{-4})^{|\Htr_{Q_\infty}(r)|} \r] - \Bigg( 1- \frac{(r+3)(6r^4+36r^3+87r^2+99r+44)}{4r^3(2r+3)^2} x^2 \\
 + \frac{(r+3)(r+1)^3(r+2)^3}{2r^5(2r+3)^2} x^3 \Bigg) \Bigg| \leq C_2 x^4 .
\end{multline*}
The Proposition follows by choosing $x$ such that $e^{-\vl/r^4} = 1-x^2 r^{-4}$.
\end{proof}

We read from \eqref{Eq_dl_of_Laplace} that
\begin{equation*}
\E\l[ |\Htr_{Q_\infty}(r)| \r] = \frac{r(r+3)(6r^4+36r^3+87r^2+99r+44)}{4(2r+3)^2} \usim r \infty \frac{3}{8} r^4 ,
\end{equation*}
which is consistent with the formula for the mean volume of the Brownian plane hull, which can be derived from \cite[Theorem 1.4]{curien2016hull}. 
\lemref{Lemma_dl_Laplace}, together with \cite[Theorem 8.1.6]{bingham1989regular}, implies that for every $r>0$, 
\begin{align}
\label{Eq_queue_distr_volume}
\P\l( |\Htr_{Q_\infty}(r)| > t \r) &\usim t \infty \frac{r(r+3)(r+1)^3(r+2)^3}{4 \sqrt \pi(2r+3)^2} t^{-3/2} .
\end{align}
Since we need a non-asymptotic upper bound valid for all $r>0$ and $t>0$, we prove the following weaker result:
\begin{corollary}
\label{Cor_moments_of_hull_volume}
For every $\vd>0$, there exists $C(\vd) \in (0,\infty)$ such that for every $r>0$, for every $t>0$,
\begin{equation*}
\P\l( |\Hull(r)| > t r^4 \r) \leq C(\vd) t^{-3/2+\vd} .
\end{equation*}
\end{corollary}

\begin{proof}
Fix $r>0$, and write $X_r = \frac{|\Htr(r)|}{r^4}$. 
By Fubini's theorem,
\begin{equation}
\label{Eq_Fubini_Laplace_moments}
\int_0^\infty \vl^{-5/2+\vd} \E\l[ e^{-\vl X_r} -1 + \vl X_r \r] \ud \vl = \E\l[ X_r^{3/2-\vd} \r] \int_{0}^\infty u^{-5/2+\vd} \l(e^{-u}+u-1\r) \ud u .
\end{equation}
Split the left-hand side integral at $\vl_0$: 
\begin{align}
\int_0^\infty \vl^{-5/2+\vd} \E\l[ e^{-\vl X_r} -1 + \vl X_r \r] \ud \vl
\label{Eq_borne_pres_0}
\leq& \int_0^{\vl_0} \vl^{-5/2+\vd} \E\l[ e^{-\vl X_r} -1 + \vl X_r \r] \ud \vl \\
\label{Eq_borne_loin_0}
&+ \E\l[ X_r \r] \int_{\vl_0}^\infty \vl^{-3/2+\vd} \ud \vl .
\end{align}
\eqref{Eq_borne_loin_0} is bounded by a constant that does not depend on $r$.  On the other hand, it follows from \lemref{Lemma_dl_Laplace} that for every $\vl \leq \vl_0$, 
\begin{equation*}
\E\l[ e^{-\vl X_r} -1 + \vl X_r \r] \leq \frac{(r+3)(r+1)^3(r+2)^3}{2r^5(2r+3)^2} \vl^{3/2} + c \vl^2 ,
\end{equation*} 
thus
\begin{align*}
&\int_0^{\vl_0} \vl^{-5/2+\vd} \E\l[ e^{-\vl X_r} -1 +\vl X_r \r] \ud \vl \\
&\leq \int_0^{\vl_0} c \vl^{-1/2+\vd} \ud \vl  + \l(\sup_{r\geq 1} \frac{(r+3)(r+1)^3(r+2)^3}{2r^5(2r+3)^2} \r) \int_0^{\vl_0} \vl^{-1+\vd} \ud \vl ,
\end{align*}
so \eqref{Eq_borne_pres_0} is bounded by a constant that depends only on $\vd$.
We then have by \eqref{Eq_Fubini_Laplace_moments}:
\begin{equation*}
\E\l[ X_r^{3/2-\vd} \r]  \leq \frac{\int_0^\infty \vl^{-5/2+\vd} \E\l[ e^{-\vl X_r} -1+ \vl X_r \r] \ud \vl}{\int_{0}^\infty u^{-5/2+\vd} \l(e^{-u}+u-1\r) \ud u}
\end{equation*}
which is smaller than a constant that depends only on $\vd$. The Lemma follows using Markov's inequality, and recalling that $|\Hull(r)| \leq |\Htr(r+1)|$.
\end{proof}

\section{Coverings of finite quadrangulations by balls}
\label{Section_Covering}

We start by proving that with high probability, we can cover the quadrangulation $Q_n$ with balls of volume uniformly bounded from below, at every scale at the same time. 

\begin{lemma}
\label{Lemma covering of Qn by balls}
Let $\ve>0$. For every integer $R$ with $n^\ve \leq R \leq n^{1/4}$, we can find a sequence $(e^R_i)_{0\leq i < R^{-4}n^{1+\ve}}$ of oriented edges of $Q_n$, such that $Q_n$ re-rooted at $e^R_i$ has the same law as $Q_n$, and the following holds with probability going to $1$ as $n\to\infty$: 

For every integer $R$ such that $n^\ve \leq R \leq n^{1/4}$, if $z^R_i$ denotes the tail vertex of $e^R_i$,
\begin{enumerate}
\item the $R$-balls $B_{Q_n}(z^R_i,R)$ cover $Q_n$,
\item for every $i$ with $0\leq i < R^{-4}n^{1+\ve}$, the ball $B_{Q_n}(z^R_i,R)$ contains at least $\frac{R^4}{n^{\ve}}$ vertices.
\end{enumerate}
\end{lemma}

In order to prove this lemma, we use the classical Cori-Vauquelin-Schaeffer bijection (or \CVS) between rooted, pointed quadrangulations and labeled trees. 
 We briefly recall it and present our notation. Let $\T_n$ be the set of all rooted labeled plane trees with $n$ edges, where by ``rooted labeled plane tree'' we mean a plane tree whose vertices bear labels in $\Z$, such that the labels of two adjacent vertices differ by at most $1$ and the root vertex has label $0$. Let $t \in \T_n$, and $l:V(t) \to \Z$ its label function. A corner of $t$ is an angular sector incident to one of its vertices. We order the corners cyclically according to the clockwise route around $t$. By convention, we extend the label function $l$ to corners of the tree, in such a way that a corner has the same label as its incident vertex. 

The \CVS allows us to get a rooted quadrangulation with $n$ faces from $t$ and from an integer $\vartheta \in \{-1,1\}$, as follows. First, add a vertex $\partial$ to the tree, and extend the labeling to $\partial$ such that $l(\partial) = -1+\inf_{t} l$. Then, for each corner $c$, let $S(c)$ be the first corner after $c$ in the contour sequence with label $l(c)-1$ (if $l(c) = \inf_t l$, we fix $S(c) = \partial$), and draw an edge between $c$ and $S(c)$ in such a way that it does not intersect the previously drawn edges. Finally, erase the edges of $t$. We obtain a quadrangulation $q$ with $n$ faces and a distinguished vertex $\partial$, and we need to specify its root edge. We root $q$ at the edge drawn from the bottom corner of the root vertex of $t$, and specify its direction using $\vartheta$: if $\vartheta = +1$ it points towards the root vertex of $t$, if $\vartheta=-1$ it points away from the root vertex of $t$.

\begin{proof}[Proof of \lemref{Lemma covering of Qn by balls}]

Let $T_n$ be uniformly chosen over $\T_n$, and $\vartheta$ a uniform integer over $\{-1,+1\}$, then the quadrangulation obtained from $(T_n,\vartheta)$ is uniformly distributed over the set of all rooted and pointed quadrangulations with $n$ faces. Forgetting the distinguished vertex, we get a uniform quadrangulation $Q_n$ with $n$ faces. 

Denote the corners of $T_n$ enumerated in clockwise order around the contour starting from the root corner by $(c_i)_{0 \leq i < 2n}$. We extend the numbering to $\Z$ by periodicity, and write $L_i$ for the label of $c_i$. \cite[Lemma 4.4]{gall2012buziosf} ensures that for every $p\geq 1$, there exists a constant $K_p$ such that for every $0 \leq j,j' \leq 2n$,
\begin{equation*}
\E\l[ \l| L_j - L_{j'} \r|^{4p} \r] \leq K_p |j-j'|^{p} .
\end{equation*}
Using Markov's inequality,
\begin{equation}
\label{Eq Label contour function is Holder}
\P\l( | L_j - L_{j'} | \geq u \r) \leq K_p \pfrac{|j-j'|}{u^4}^p .
\end{equation}
Define
\begin{equation}
\AA \eqdef \l\{ \forall i,j \in \{0, ..., 2n-1\} , i\neq j \ : \ |L_i - L_j| < n^{1/p} |i-j|^{1/4} \r\} .
\end{equation}
Then
\begin{align*}
\P(\AA^c) &\leq \sum_{i,j \in \{0, ..., 2n-1\}, i \neq j} K_p \pfrac{|i-j|}{n^{4/p} |i-j|}^p \\
&\leq 4 K_p n^{-2} .
\end{align*}
Fix $p > 5/\ve$, and let us argue on $\AA$. Write $v_i$ for the vertex of $Q_n$ that is incident to $c_i$. \cite[Proposition 5.9 (i)]{gall2012buziosf} allows us to bound the distance in $Q_n$ between $v_i$ and $v_j$ for $i\leq j$: 
\begin{equation*}
\dgr^{Q_n}(v_i,v_j) \leq L_i + L_j - 2\min_{k \in [i,j]} L_{k} +2 \leq 2 n^{1/p} |i-j|^{1/4} +2 . 
\end{equation*}
For every $0 \leq i, j < 2n$, $v_j$ thus belongs to the ball of radius $2+2n^{1/p}|i-j|^{1/4}$ centered at $v_i$. For every $n^\ve \leq R \leq n^{1/4}$, define $k(R)>0$ as the largest integer such that $2+2n^{1/p} k(R)^{1/4} < R-2$, and fix $y^R_i = v_{k(R) i}$ for every $0 \leq i < \ceil{2n/k(R)}$. Every vertex of $Q_n$ (except possibly $\partial$) is at distance strictly less than $R-2$ from at least one $y^R_i$, and $\partial$ is at distance at most $R-2$ from one of $y^R_i$, thus the balls $B_{Q_n}(y^R_i,R-1)$ cover $Q_n$. For $n$ large enough, this covering of $Q_n$ contains at most $2n/k(R) \leq n^{1+5/p}/R^4$ balls. 

Note that for every $n^\ve \leq R \leq n^{1/4}$ and $0 \leq i < \ceil{2n/k(R)}$, if we re-root the tree at the corner $c_{k(R)i}$ and subtract $l(c_{k(R)i})$ to the labels of $t$ to ensure that the resulting tree is in $\T_n$, then the map we obtain by the \CVS is exactly $Q_n$, re-rooted at the edge $e^R_{i}$ drawn from corner $c_{k(R)i}$ (oriented towards $y^R_i$ if $\vartheta=+1$, and away from it if $\vartheta=-1$). In particular, it has the same law as $Q_n$.
We complete the sequences $(e^R_i)$ and $y^R_i$ by taking $e_i^R$ equal to the root edge and $y^R_i$ equal to the root vertex for every $\ceil{2n/k(R)} \leq i < n^{1+\ve}/R^4$. Let $z^R_i$ be the tail vertex of $e^R_{i}$. Since for every $0 \leq i < n^{1+\ve}/R^4$, $z^R_i$ is at distance at most $1$ from $y^R_i$, the first property of the lemma holds.

It remains to prove that the volume of every $B_{Q_n}(y^R_i,R)$ is bounded from below by $R^4/n^{\ve}$ for every $n^\ve \leq R \leq n^{1/4}$ and $0 \leq i < n^{1+\ve}/R^4$. 
From now on, we argue on the intersection of $\AA$ with the event where the maximum degree of $Q_n$ is at most $\ln n$, which has probability going to $1$ as $n \to \infty$ by \lemref{Lemma_max_deg_Qn}. The number of corners of a vertex $v_i$ in $T_n$ is at most its degree in $Q_n$, so it is smaller than $\ln n$. 
For every $n^\ve \leq R \leq n^{1/4}$ and $0 \leq i < \ceil{2n/k(R)}$, every $v_j$ with $k(R)i \leq j \leq k(R)(i+1)$ belongs to the $R$-ball $B_{Q_n}(z^R_i,R)$. Each $v_j$ appears at most $\ln n$ times in this sequence, so that $B_{Q_n}(z^R_i,R)$ contains at least $\frac{k(R)+1}{\ln n} \geq \frac{R^4}{n^{\ve}}$ distinct vertices. Since this also holds for $\ceil{2n/k(R)} \leq i < n^{1+\ve}/R^4$, the second point of the lemma is proven. 

\end{proof}

The following proposition is the key ingredient of the proof of Theorem \ref{Th_isoperimetric_inequality}. Recall that $r$-hulls are obtained from $r$-balls by adding every connected component of the complement of the $r$-ball but the one containing the largest number of faces. Proposition \ref{Prop_covering_hulls_with_controlled_volume} strengthens \lemref{Lemma covering of Qn by balls} by exhibiting coverings of $Q_n$ by balls, such that the volumes of the corresponding hulls are bounded from above and below.

\begin{proposition}
\label{Prop_covering_hulls_with_controlled_volume}
For every $\vd \in (0,1/8)$, the following holds with probability going to $1$ as $n$ goes to $\infty$: 

For every $R$ of the form $2^k$ with $n^\vd \leq R \leq n^{1/4-\vd}$, we can find a sequence $z^R_i$, $0 \leq i <\frac{n^{1+\vd}}{R^4}$ of vertices of $Q_n$ such that 
\begin{enumerate}
\item the balls $B_{Q_n}(z^R_i,R/2)$ cover $Q_n$,
\item for every $i$, $\frac{R^4}{n^{\vd}} \leq |\Hull_{Q_n}(z^R_i,R)| \leq R^{4/3} n^{2/3+\vd}$.
\end{enumerate}
\end{proposition}

\begin{proof}

Let $\ve = \vd/40$. 
By \cite{chassaing2004random}, the probability that the diameter of $Q_n$ is at least $8 n^{1/4-\ve}$ goes to $1$ as $n$ goes to infinity. From now on we argue on the intersection of this event with the event of \lemref{Lemma covering of Qn by balls}, whose probability also goes to $1$ as $n$ goes to infinity. We consider the sequence $(z^i_R)$ given by \lemref{Lemma covering of Qn by balls}, so the property 1 and the minoration in the property 2 of  \propref{Prop_covering_hulls_with_controlled_volume} already hold. 

For every $0 < R \leq n^{1/4-\ve}$, for every $x \in V(Q_n)$, one may find $z \in V(Q_n)$ at distance at least $4n^{1/4-\ve}$ from $x$. Since we work of the event of \lemref{Lemma covering of Qn by balls}, this vertex is contained in a $\ceil{n^{1/4-\ve}}$-ball $B$ containing at least $n^{1-5\ve}$ vertices. 
Consider now the ball $B_{Q_n}(x,R)$. One of the connected component of its complement contains the ball $B$, thus it contains at least $n^{1-5\ve}$ inner vertices. Since this connected component is a quadrangulation with simple boundary, Euler's formula gives that its number of faces is also larger than $n^{1-5\ve}$. It follows that the $R$-hull $\Hull_{Q_n}(x,R)$ contains at most $n-n^{1-5\ve}$ faces.

Now take $n^\ve \leq R \leq n^{1/4-\ve}$ and $0\leq i< \frac{n^{1+\ve}}{R^4}$. The event $\{ |\Hull_{Q_\infty}(\rho_\infty,R)| > t R^4 \}$ has probability at most $C(\ve) t^{-\frac{3}{2}+\ve}$ by Corollary \ref{Cor_moments_of_hull_volume}. Together with \lemref{Lemma_density_hulls_UIPQ_Qn}, and using the fact that $Q_n$ re-rooted at $e^i_R$ has the same law as $Q_n$, 
we get
\begin{equation*}
\P( |\Hull_{Q_n}(z^R_i,R)| > t R^4 ) \leq c \l(n^{-5\ve}\r)^{-5/2} \P( |\Hull_{Q_\infty}(\rho_\infty,R)| > t R^4 ) \leq cC(\ve) n^{13\ve} t^{-\frac{3}{2}+\ve}.
\end{equation*}
Fix $t = \pfrac{R^4}{n^{1+15\ve}}^{\frac{1}{-\frac{3}{2}+\ve}}$ and sum over every $0 \leq i < n^{1+\ve}/R^4$:
\begin{align}
\label{Eq_hulls_ont_grand_volume_faible_proba}
\P\l( \exists \ 0\leq i< \frac{n^{1+\ve}}{R^4} \ : \ |\Hull_{Q_n}(z^R_i,R)| > t R^4 \r) &\leq \frac{n^{1+\ve}}{R^4} c C(\ve) n^{13\ve} \pfrac{R^4}{n^{1+15\ve}} \\
&\leq c C(\ve) n^{-\ve} . \nonumber
\end{align}

Finally, consider the union of the events in \eqref{Eq_hulls_ont_grand_volume_faible_proba} over all $R$ with $n^{\ve} \leq R \leq n^{1/4-\ve}$, such that $R$ is of the form $2^k$ for some integer $k$. 
There are at most $(\ln n)/(\ln 2)$ such $R$, so the probability that the event in \eqref{Eq_hulls_ont_grand_volume_faible_proba} holds for at least one of these $R$ goes to $0$ as $n$ goes to infinity. Since $tR^4 \leq R^{4/3} n^{2/3+\vd}$, this gives the majoration in property 2.

\end{proof}

\section{Proof of the bound on the size of bottlenecks}
\label{Sec_Isoperimetric inequality}

We now use the results of Section \ref{Section_Covering} to prove \thref{Th_isoperimetric_inequality}. 
We first establish the inequality of \thref{Th_isoperimetric_inequality} for sets of faces of the quadrangulation whose boundary is connected, then for sets of faces that are connected when seen as a closed subdomain of the sphere (we simply say ``connected'' from now on), before proving it for any set of faces. The first step is done in the following Lemma:

\begin{lemma}
\label{Lemma_isoperimetric_inequality}
For every $\nu \in (0,3/8)$, the probability that
\begin{equation*}
\inf \frac{|\partial S|^{4/3}}{|S|} \geq n^{-2/3-\nu} 
\end{equation*}
goes to $1$ as $n \to \infty$, where the infimum holds over all subsets $S$ of $F(Q_n)$ such that $|S| \leq n/2$ and $\partial S$ is connected. 
\end{lemma}

\begin{proof}
Let us argue on the event of \propref{Prop_covering_hulls_with_controlled_volume} with $\vd = \nu/3$. 
Consider $S$ a subset of $F(Q_n)$ with $|S| \leq n/2$ such that $\partial S$ is connected. 

If $|\partial S| \geq n^{1/4-\vd}$, since we necessarily have $|S| \leq n$, 
\begin{equation*}
\frac{|\partial S|^{4/3}}{|S|} \geq \frac{n^{\frac{4}{3}\l(\frac{1}{4}-\vd\r)}}{n} = n^{-2/3-4\vd/3} \geq n^{-2/3-\nu} .
\end{equation*}

If $|\partial S| < n^{\vd}$, let $R$ the smallest power of two larger than $2n^\vd$. $\partial S$ is contained in a $R$-ball centered at one of $z^R_i$. Either $S$ is contained in $\Hull_{Q_n}(z^R_i,R)$, or $S$ contains the complement of $\Hull_{Q_n}(z^R_i,R)$. 
We are on the event of \propref{Prop_covering_hulls_with_controlled_volume}, so (by the second point in the proposition) the complement of $\Hull_{Q_n}(z^R_i,R)$ has volume strictly larger than $n/2$. Since $|S| \leq n/2$, 
$S$ is contained in $\Hull_{Q_n}(z^R_i,R)$, and thus its number of faces is less than $n^{\frac{2}{3}+\vd+\frac{4}{3}\vd}$ (by property 2 of \propref{Prop_covering_hulls_with_controlled_volume}), which gives 
\begin{equation*}
\frac{|\partial S|^{4/3}}{|S|} \geq n^{-2/3-3\vd} \geq n^{-2/3-\nu}.
\end{equation*}

If $n^{\vd} \leq |\partial S| < n^{1/4-\vd}$, consider $R = 2^k$ with $R/2 \leq |\partial S| < R$. 
Since the balls $B_{Q_n}(z^R_i,R/2)$ cover $Q_n$, we can find $i$ such that $d(z^R_i, \partial S) \leq R/2$, hence by connexity $\partial S \subset B_{Q_n}(z^R_i,R)$. By the same argument as in the case $|\partial S| < n^{\vd}$, we have $S \subset \Hull_{Q_n}(z^R_i,R)$, and thus
\begin{equation*}
\frac{|\partial S|^{4/3}}{|S|} \geq \frac{(R/2)^{4/3}}{R^{4/3} n^{2/3+\vd}} \geq n^{-2/3-\nu}.
\end{equation*}

\end{proof}

\begin{proof}[Proof of \thref{Th_isoperimetric_inequality}]

Let $\vd>0$. 
Fix $\nu = 3\vd/4$, and let $n$ be taken large enough so that $n^{\nu/3} > 2$. For the rest of the proof, we argue on the event of \lemref{Lemma_isoperimetric_inequality}, whose probability goes to one as $n \to \infty$.

Consider $S$ a connected subset of $F(Q_n)$ with $|S| \leq n/2$ and $|\partial S|\leq n^{1/4-\nu}$. Let $S_1, \ldots, S_k$ be the connected components of the complement of $S$. Note that every $S_i$ is connected and its complement is connected, so by planarity considerations $\partial S_i$ is connected too; furthermore, the $\partial S_i$ are disjoint, and $\partial S = \cup_{1 \leq i \leq k} \partial S_i$. We claim that if $n$ is large enough, then exactly one of the $S_i$ has volume at least $n/2$. Note that it is equivalent to show that at least one has volume at least $n/2$. We prove this claim by contradiction: assume that every $S_i$ has volume at most $n/2$. Since we work on the event of \lemref{Lemma_isoperimetric_inequality}, for every $1\leq i \leq k$, $|S_i| \leq n^{2/3 + \nu} |\partial S_i|^{4/3}$, thus
\begin{align*}
n = |Q_n| &= |S| + \sum_{i=1}^k |S_i|  \\
&\leq n/2 + \sum_{i=1}^k n^{2/3 + \nu} |\partial S_i|^{4/3} \\
&\leq n/2 + n^{2/3 + \nu} |\partial S|^{4/3} \\
&\leq n/2 + n^{1 - \nu/3} \\
&< n 
\end{align*}
which is impossible, proving our claim. 
Without loss of generality, we assume that $|S_1| > n/2$. Define $S' \eqdef S \cup \bigcup_{2 \leq i \leq k} S_i$. Then $\partial S' = \partial S_1$ is connected and $|S'| \leq n/2$, thus 
\begin{equation*}
|S| \leq |S'| \leq n^{2/3 + \nu} |\partial S'|^{4/3} \leq n^{2/3 + \nu} |\partial S|^{4/3} .
\end{equation*} 
It remains to consider the case of a connected $S \subset F(Q_n)$ with $|S|\leq n/2$ and $|\partial S| > n^{1/4-\nu}$. It is immediate that for such a set, 
\begin{equation*}
|S| \leq n \leq n^{2/3 + 4\nu/3} |\partial S|^{4/3} .
\end{equation*} 
To sum up, on the event of \lemref{Lemma_isoperimetric_inequality}, for every connected set $S \subset F(Q_n)$ with $|S|\leq n/2$,
\begin{equation}
\label{Eq_ineg_isop_connected}
|S| \leq n^{2/3 + \vd} |\partial S|^{4/3} .
\end{equation}

Now consider a generic $T \subset F(Q_n)$ with $|T|\leq n/2$. Let $T_1, ..., T_j$ be its connected components. By \eqref{Eq_ineg_isop_connected}, for every $1 \leq i \leq j$, $|T_i| \leq n^{2/3 + \vd} |\partial T_i|^{4/3}$, and since $\partial T$ is the disjoint union of the $\partial T_i$ for $1\leq i \leq j$, by convexity
\begin{equation*}
|T| = \sum_{i=1}^j |T_i| \leq \sum_{i=1}^j n^{2/3 + \vd} |\partial T_i|^{4/3} \leq n^{2/3 + \vd} |\partial T|^{4/3} .
\end{equation*}
\thref{Th_isoperimetric_inequality} follows.

\end{proof}

\section{Simulations}

The upper bound on the uniform mixing time of Theorems \ref{Th_mixing_time} and \ref{Th_mixing_time_dual}, in $n^{3/2+o(1)}$, does not match the lower bound in $n^{1+o(1)}$ that we derived from heuristic considerations in the introduction. To try and conjecture what the correct asymptotic is, we have simulated quadrangulations of size ranging from 10 to 2500, and computed the $\frac{1}{2}$-uniform mixing time of the lazy random walk on their vertices, as well as the $\frac 1 2$-mixing time in total variation and the relaxation time, two usual notions of mixing times, defined as follows for a Markov chain with state space $E$, transition kernel $P$ and stationary measure $\pi$:
\begin{equation}
\begin{cases}
\tau^{\mathrm{TV}}(\ve) = \inf\l\{ n\geq 1\ : \ \sup_{x\in E} d_\mathrm{TV}(P^n(x,\cdot), \pi) < \ve \r\} ,\\
\tau^\mathrm{rel} = (1-\vl_2)^{-1} ,
\end{cases}
\end{equation}
where $\vl_2$ is the second-highest eigenvalue of $P$ (the highest being 1). We also computed the uniform mixing time, resp. the mixing time in total variation and the relaxation time, for the lazy random walk on the faces. The simulation code, in R, is available on the author's webpage.\footnote{\url{https://www.math.uzh.ch/index.php?id=people&key1=12738}} 

Quadrangulations with $n$ vertices are generated using the Cori-Vauquelin-Schaeffer bijection, in linear time. The uniform mixing time is computed by quick exponentiation of the transition matrices of the lazy random walk. This last step consumes the bulk of the computation time, and led us to only simulate quadrangulations with up to 2500 vertices, see \figref{Fig_number_simulations}. 

\begin{figure}[!h]
\begin{center}
\begin{tabular}{c|c|c|c|c|c|c|c|c|c}
number of vertices & 10 & 20 & 40 & 80 & 160 & 320 & 640 & 1280 & 2500 \\
\hline
number of simulations & 40000 & 40000 & 40000 & 40000 & 40000 & 26000 & 5000 & 1940 & 457
\end{tabular}
\end{center}
\caption{Number of simulated maps for each selected sizes.}
\label{Fig_number_simulations}
\end{figure}

Our first observation is that the distribution of the uniform mixing times (renormalized by their empirical means) seems to converge as the size of the map goes to $\infty$, as illustrated in \figref{Fig_density_mixing_time_sample_multiple}. 
This holds true as well for the mixing time in total variation and the relaxation time, for the lazy random walk on vertices and on faces. 
The limit law presents a light tail near zero, as well as an exponential tail towards $+\infty$.

\begin{figure}[!h]
\begin{center}
\includegraphics[width=\textwidth]{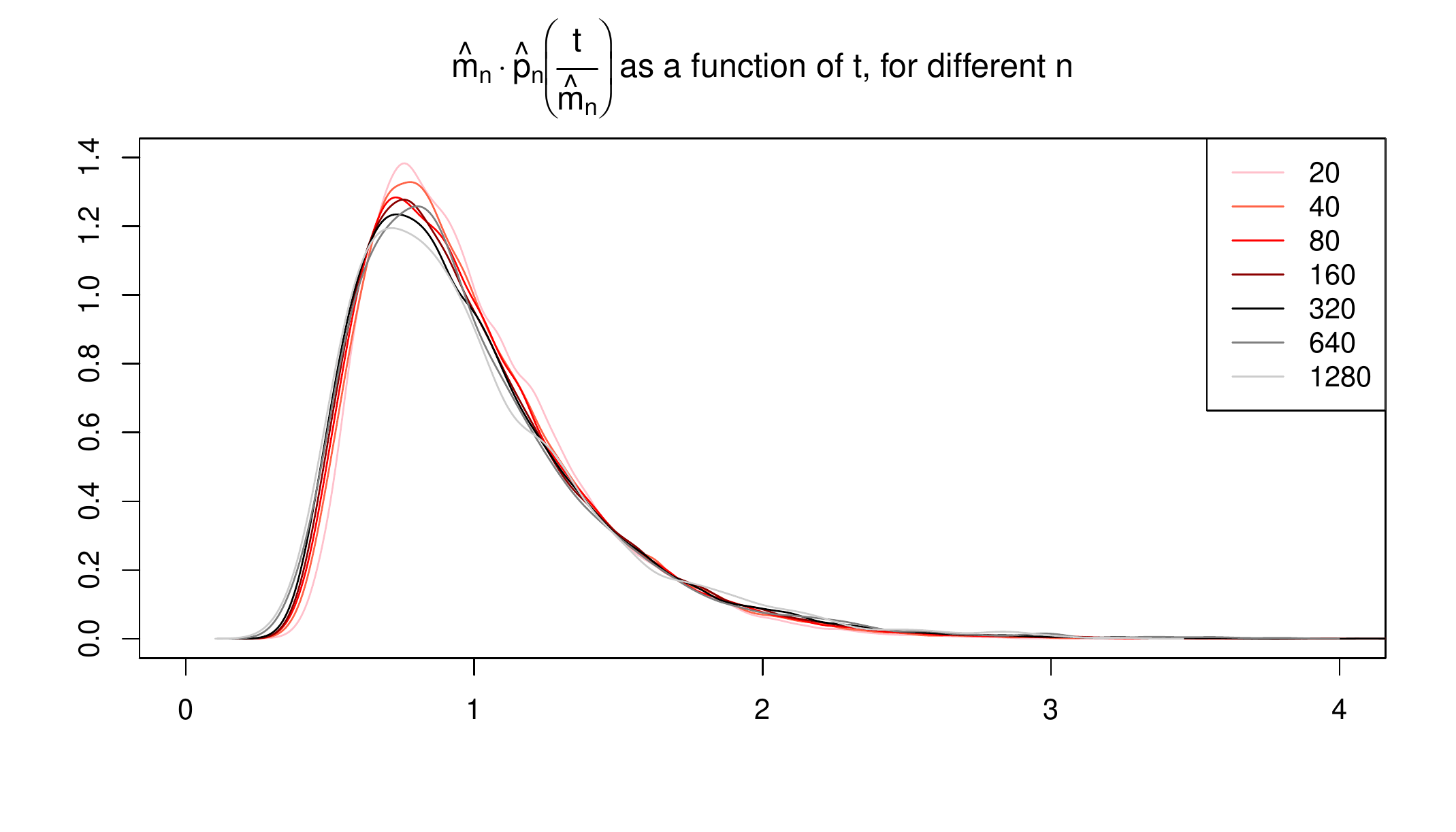}
\end{center}
\vspace{-3em}
\caption{Estimated probability mass function $\hat p_n$ of $\tau_{Q_n}(1/2)$, for various values of $n$, rescaled by the empirical mean $\hat m_n$ of $\tau_{Q_n}(1/2)$.} 
\label{Fig_density_mixing_time_sample_multiple}
\end{figure}

We also note a strong correlation between different mixing times, see \figref{Fig_covariance_mixing_times}. 
In particular, the mixing times in the quadrangulation and its dual are increasingly correlated as the size of the quadrangulation increases; we draw a parallel with \cite{fpp}, where the authors prove that triangulations and their duals are asymptotically isometric in the large scale as their size goes to $+\infty$ (the results can be generalized to quadrangulations, as is done in \cite{lehe2019fpp}, although it does not handle the dual). This leads us to conjecture that it is the macroscopic scale, rather than the microscopic scale, that influences the mixing time the most, in the sense that if two maps are asymptotically isometric when their distances are appropriately rescaled, then their mixing times will be asymptotically equal.

\begin{figure}[!h]
\begin{center}
\begin{tabular}{c|c|c|c|c}
number of vertices & $\mathrm{Corr}(\tau_{Q_n}, \tau^\mathrm{TV}_{Q_n})$ & $\mathrm{Corr}(\tau_{Q_n}, \tau^\mathrm{rel}_{Q_n})$ & $\mathrm{Corr}(\tau^\mathrm{TV}_{Q_n}, \tau^\mathrm{rel}_{Q_n})$ & $\mathrm{Corr}(\tau_{Q_n}, \tau_{Q^\dagger_n})$ \\ \hline
10  &  0.9610  &  0.8521 & 0.9375  &  0.5383 \\ \hline
20  &  0.9638 &  0.8678  & 0.9576 &  0.5834 \\ \hline
40  &  0.9610 &  0.8673  & 0.9618 & 0.6542 \\ \hline
80  &  0.9588  &  0.8683  & 0.9641 & 0.7185 \\ \hline
160  &  0.9565  &  0.8653  & 0.9645 & 0.7766 \\ \hline
320  &  0.9557  &  0.8655  & 0.9651 & 0.8207 \\ \hline
640  &  0.9547  &  0.8697  & 0.9684 & 0.8625 \\ \hline
1280  &  0.9530  &  0.8652  & 0.9673 & 0.8953 \\ \hline
2500  &  0.9453  &  0.8466  & 0.9639 & 0.8995 
\end{tabular}
\end{center}
\caption{Empirical correlation coefficients, defined as the ratio of the empirical covariance by the product of the empirical standard deviations, of several pairs of mixing times, as a function of the number of vertices of the quadrangulation. The level of the mixing time is always 1/2 (when relevant).}

\label{Fig_covariance_mixing_times}
\end{figure}

Finally, the simulations give some insight into the asymptotic of the mixing time. It appears that the conjectured lower bound $\tau_{Q_n}(1/2) \geq n^{1+o(1)}$ is closer to the truth. 
\figref{Fig_asymp_ratio_mixing_time_sample} shows the empirical mean of $\frac{\tau_{Q_n}(1/2)}{n}$ as a function of $\ln n$. The fact that this quantity increases in $n$ supports the lower bound. As for the upper bound, if $\tau_{Q_n}(1/2) = n^{\vg + o(1)}$ for some $1 < \vg \leq 3/2$, then the sequence in \figref{Fig_asymp_ratio_mixing_time_sample} should grow exponentially: the concavity of the sequence seems to prevent this claim. In fact, assuming that this sequence remains concave, or more weakly that it grows at most linearly, would directly yield that $\tau_{Q_n}(1/2) = O(n \ln n)$. We formalize our observation in the following conjecture. 

\begin{figure}[!ht]
\begin{center}
\includegraphics[width=0.7\textwidth]{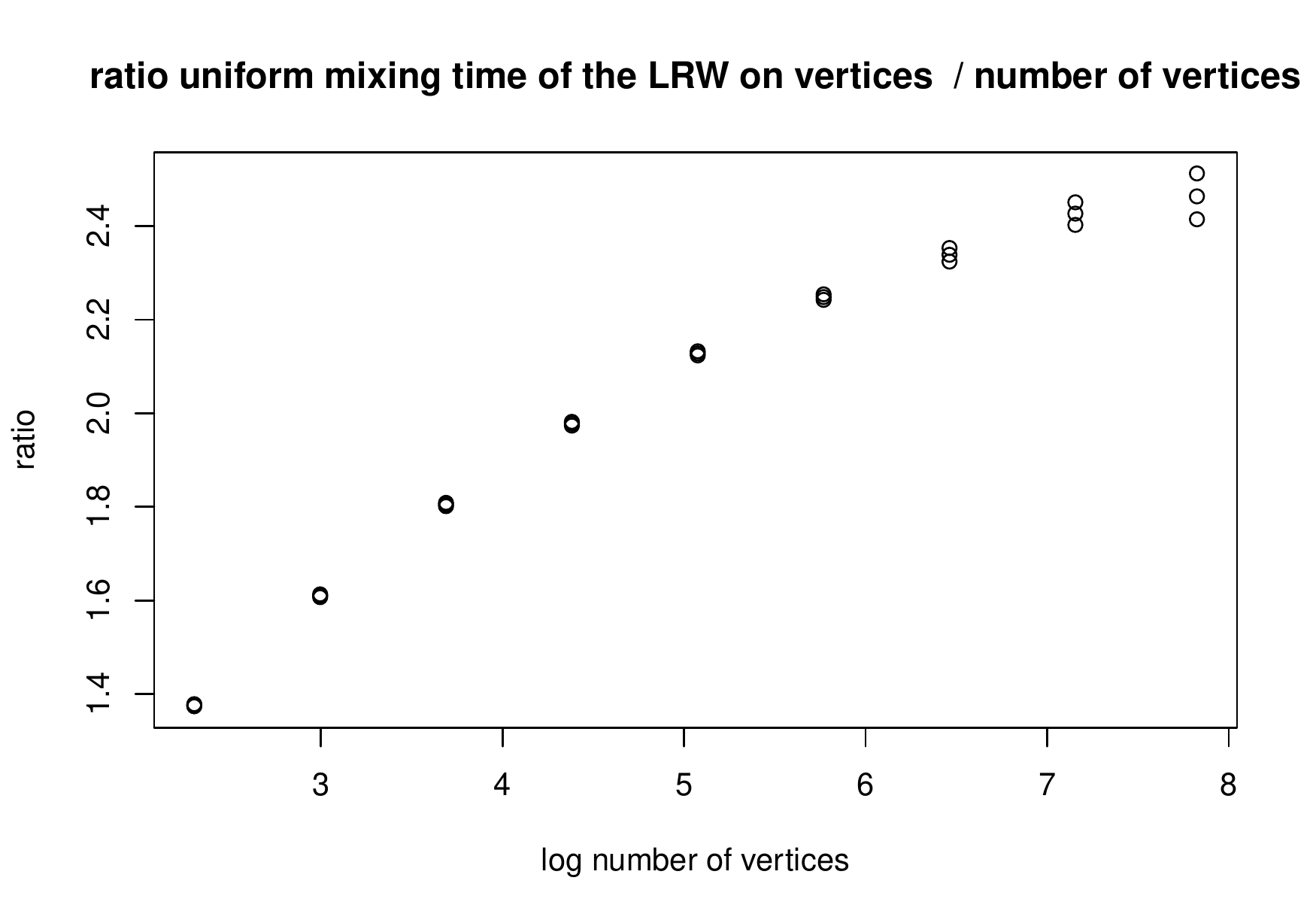}
\end{center}
\vspace{-1em}
\caption{$\frac{\tau_{Q_n}(1/2)}{n}$ as a function of $\ln n$. We display three points for each value of $n$: the central one is the empirical mean of the observations, and the top and bottom ones are at $\pm 1$ empirical standard deviation from the empirical mean. More mixing times are represented in \figref{Fig_asymp_ratio_mixing_time} in Annex.}
\label{Fig_asymp_ratio_mixing_time_sample}
\end{figure}

\begin{conjecture}

For every $\vd, \ve > 0$, with probability going to $1$ as $n\to\infty$, $\tau_{Q_n}(\ve) \in [n^{1-\vd}, n^{1+\vd}]$. This also holds for $\tau_{Q^\dagger_n}(\ve)$, and for $\tau^\mathrm{TV}_{Q_n}(\ve)$, $\tau^\mathrm{TV}_{Q^\dagger_n}(\ve)$,  $\tau^\mathrm{rel}_{Q_n}$ and  $\tau^\mathrm{rel}_{Q^\dagger_n}$.
\end{conjecture}

This conjecture could be further strengthened by specifying the existence of a sequence $r_n$ such that $r_n^{-1} \tau_{Q_n} (\ve)$ converges in distribution, in accordance with our above observation.

Another possibility of continuation would be to consider the simple random walk instead of the lazy random walk: the mixing time of the simple random walk should asymptotically be half the mixing time of the lazy random walk (of course one needs to be careful when working on bipartite graphs since they are not aperiodic). One could also work on other models of random maps; in fact, we feel confident that the methods in this article can be adapted to triangulations.

\bibliographystyle{plain}
\bibliography{bibliographie_tpsMel}

\newpage

\section*{Annex: Illustration of the mixing times}

We plot the estimated density, or rather the estimated probability mass function or PMF, of various mixing times for the uniform quadrangulation with $n=320$ vertices. The choice of $n=320$ yields a good compromise between a large $n$ to be more faithful to a possible “limit shape”, and a large number of observations for a better estimations.

\begin{figure}[!h]
\begin{center}
\includegraphics[width=0.45\textwidth]{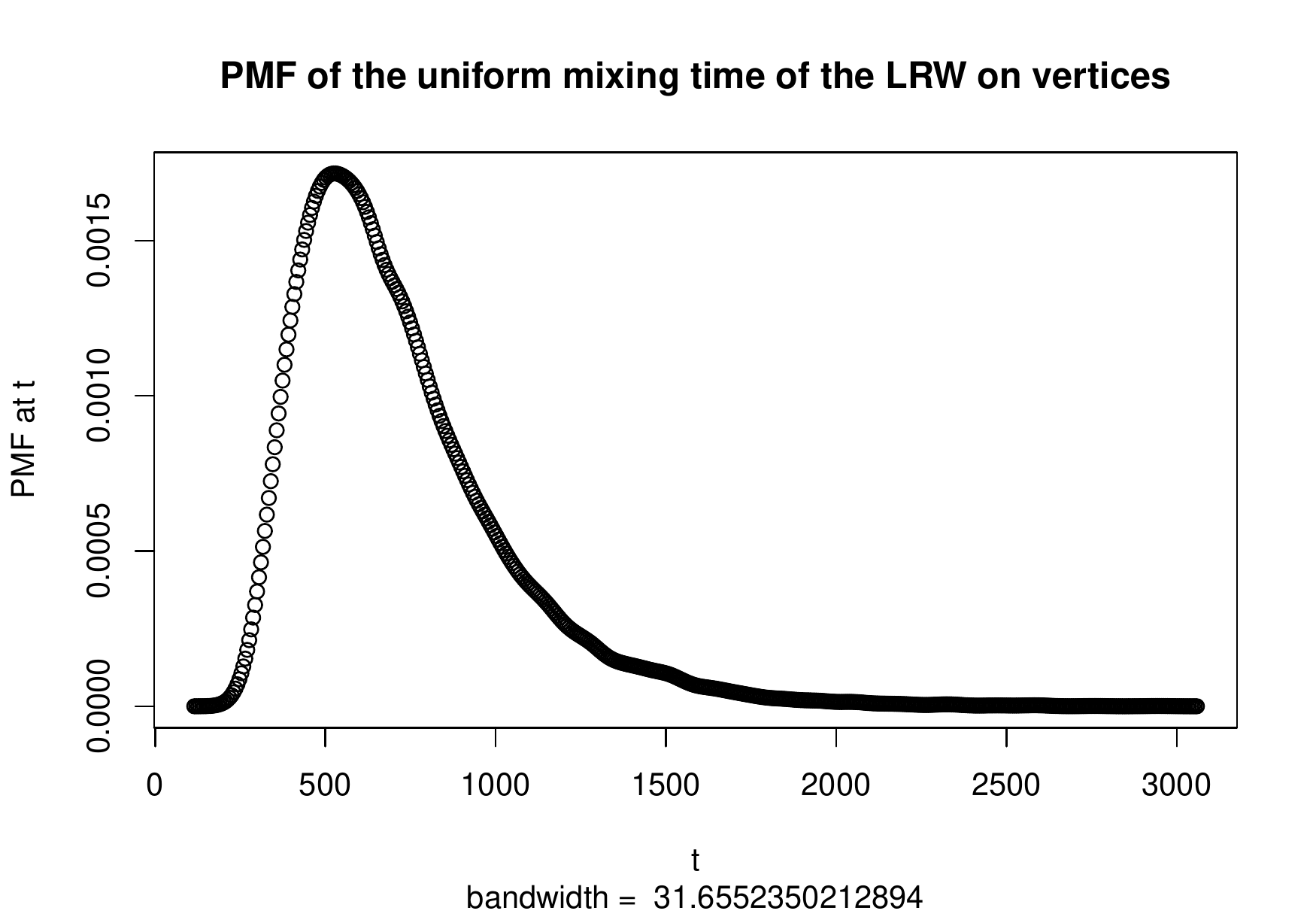}
\includegraphics[width=0.45\textwidth]{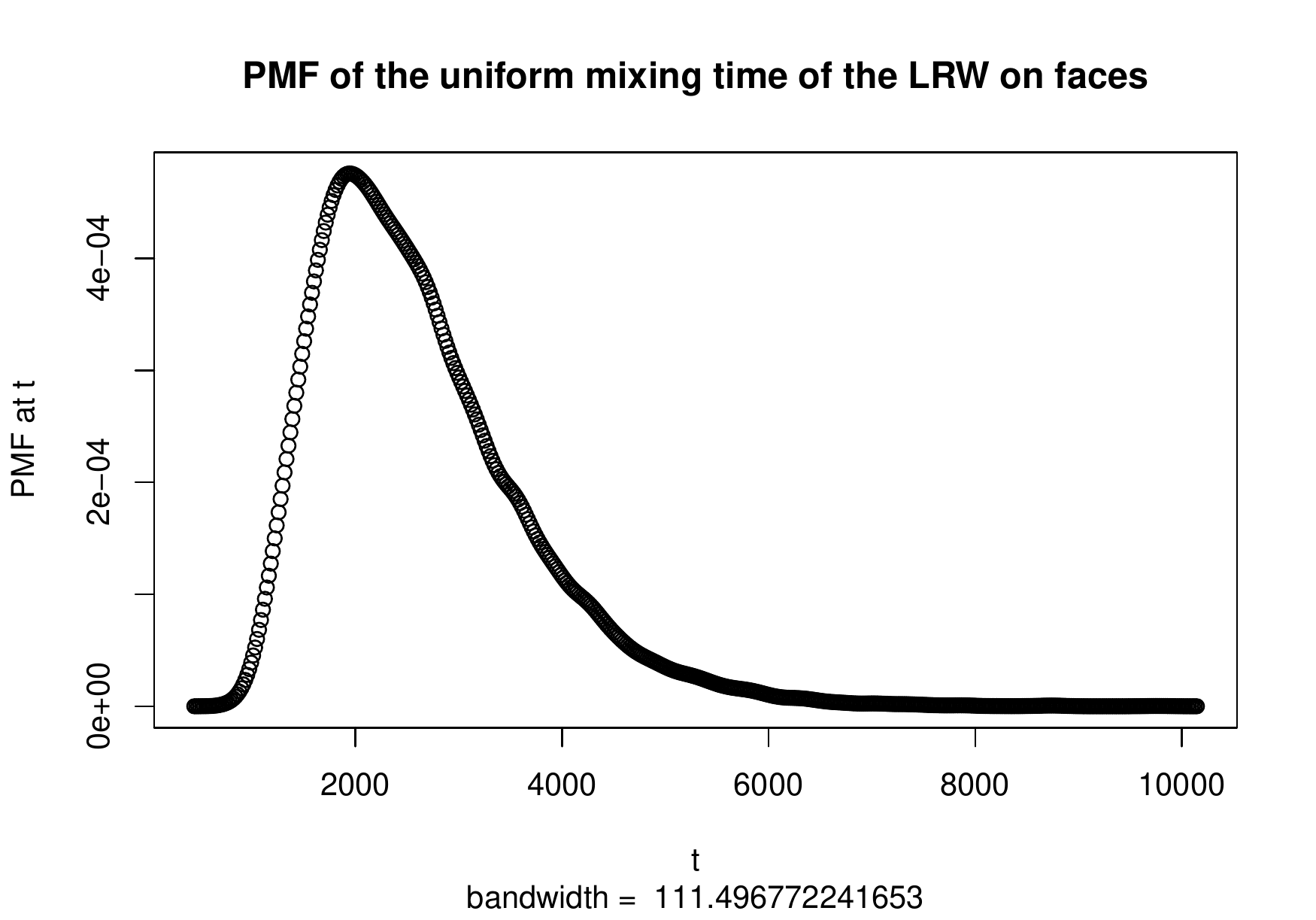}

\includegraphics[width=0.45\textwidth]{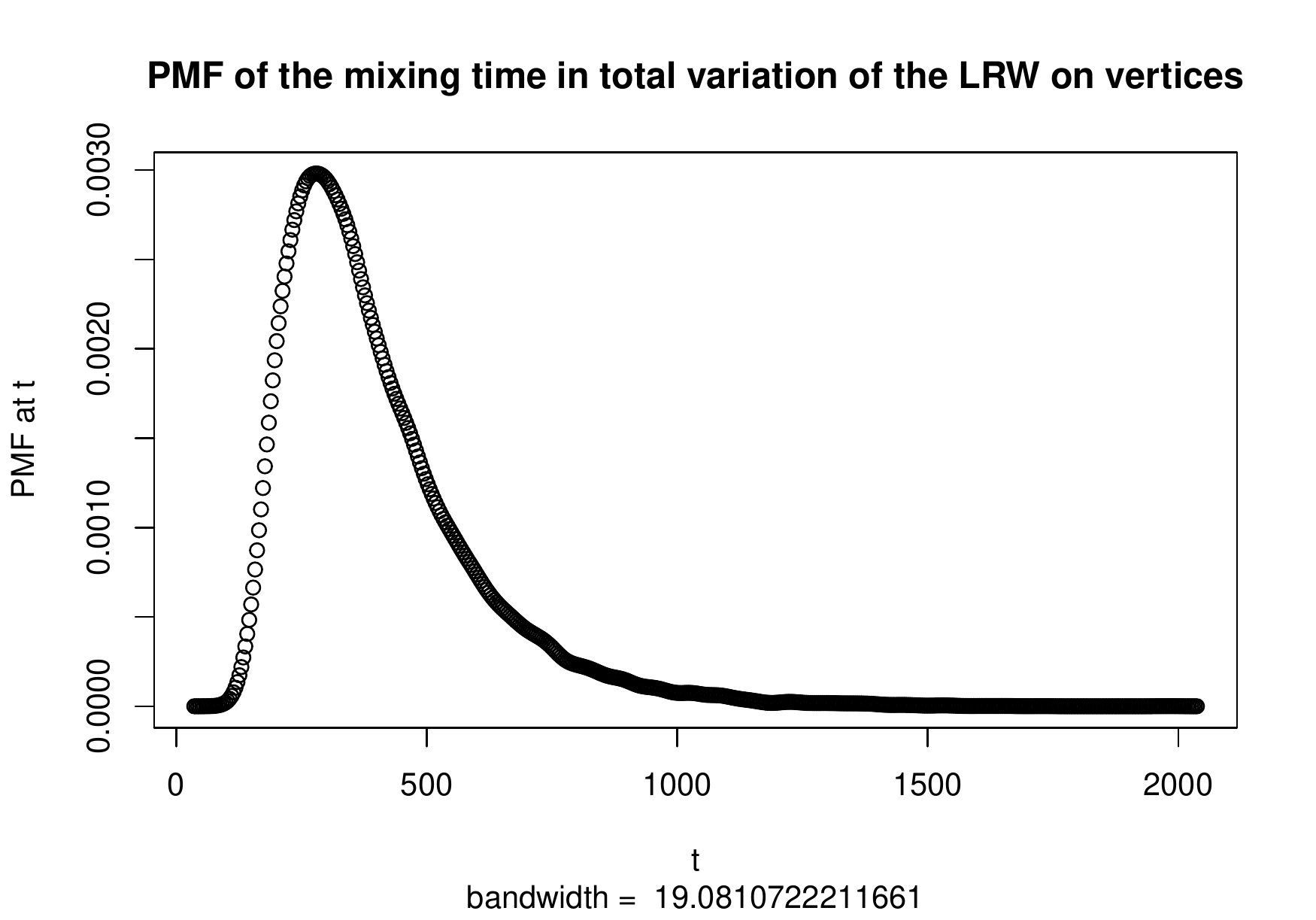}
\includegraphics[width=0.45\textwidth]{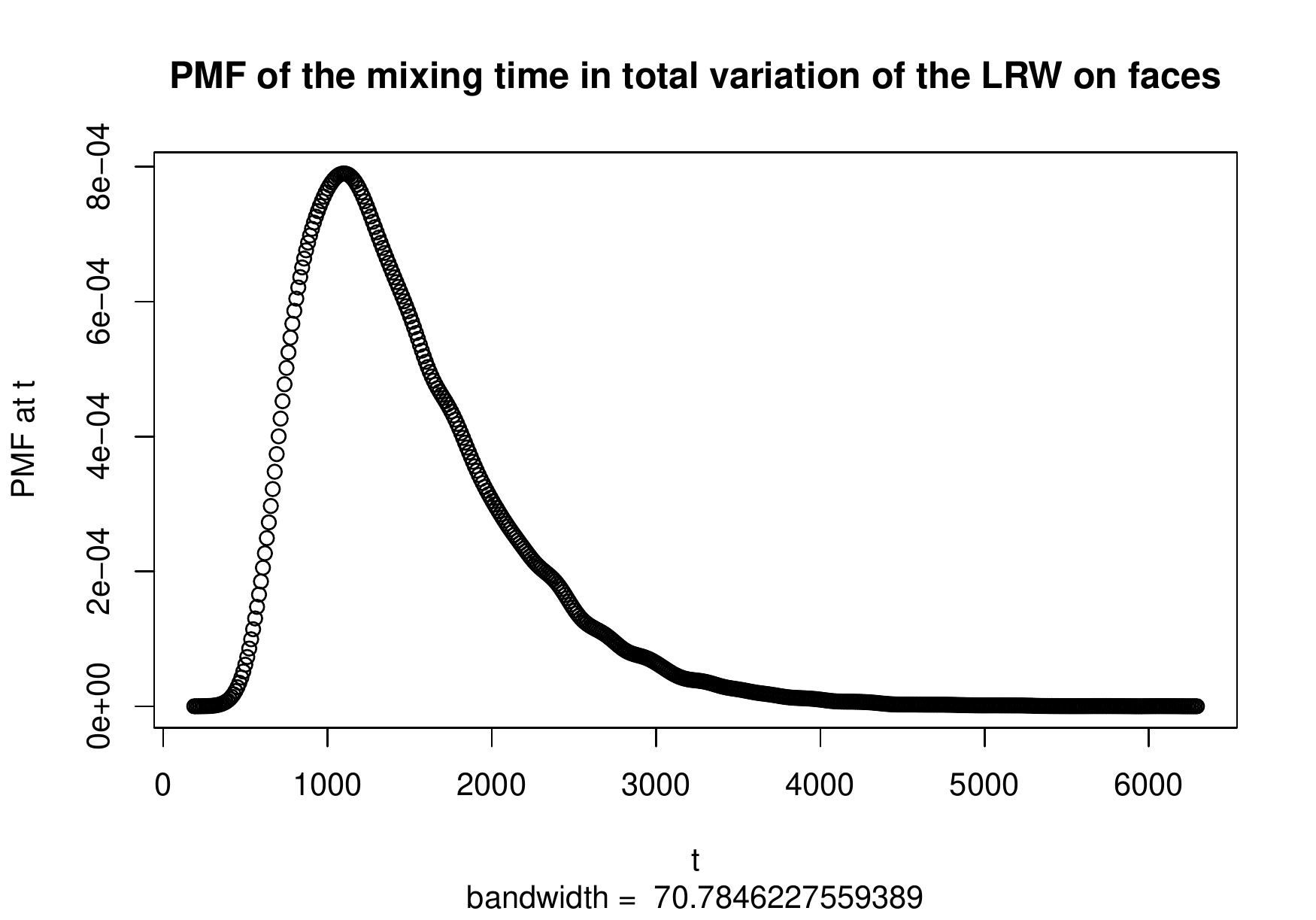}

\includegraphics[width=0.45\textwidth]{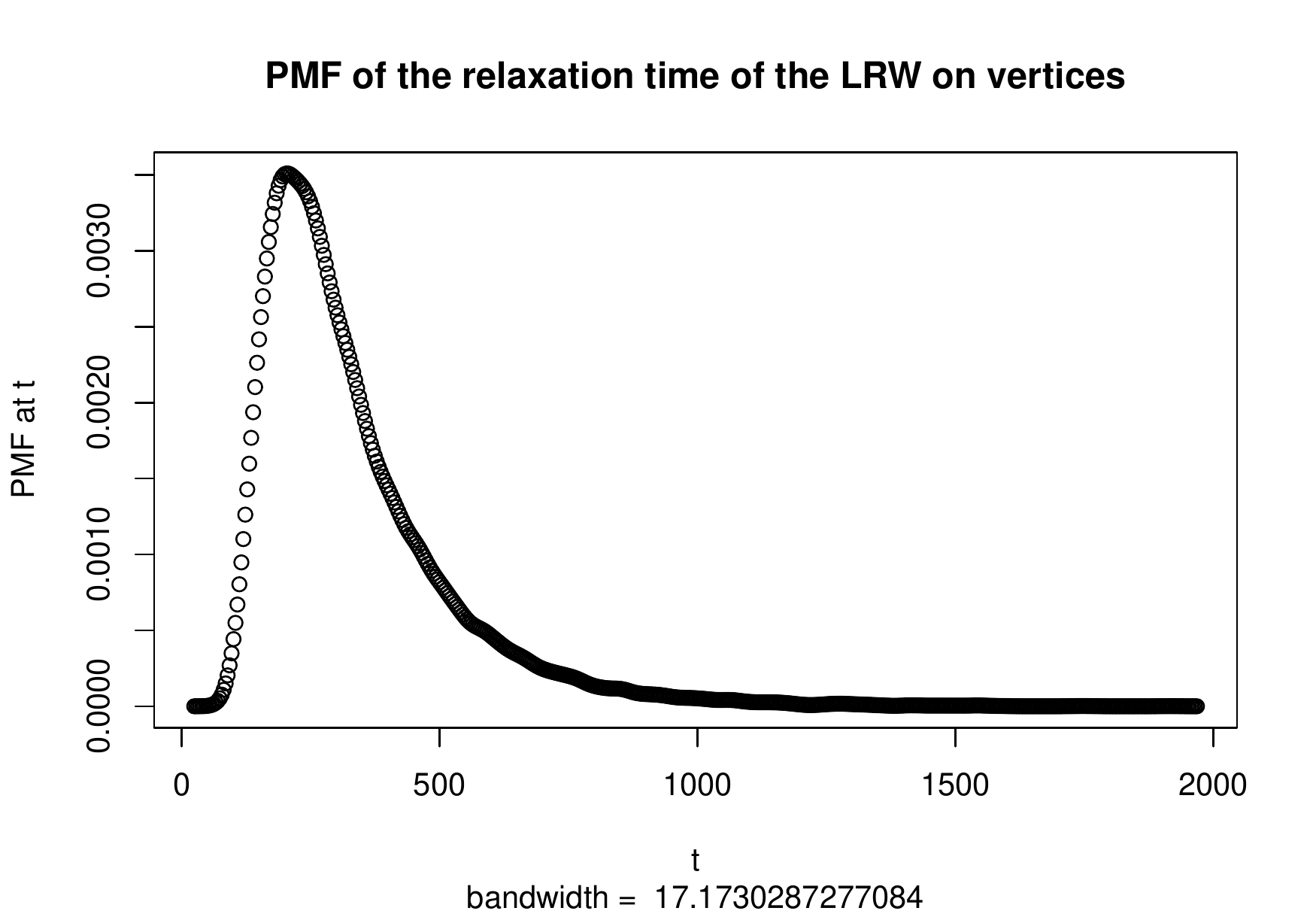}
\includegraphics[width=0.45\textwidth]{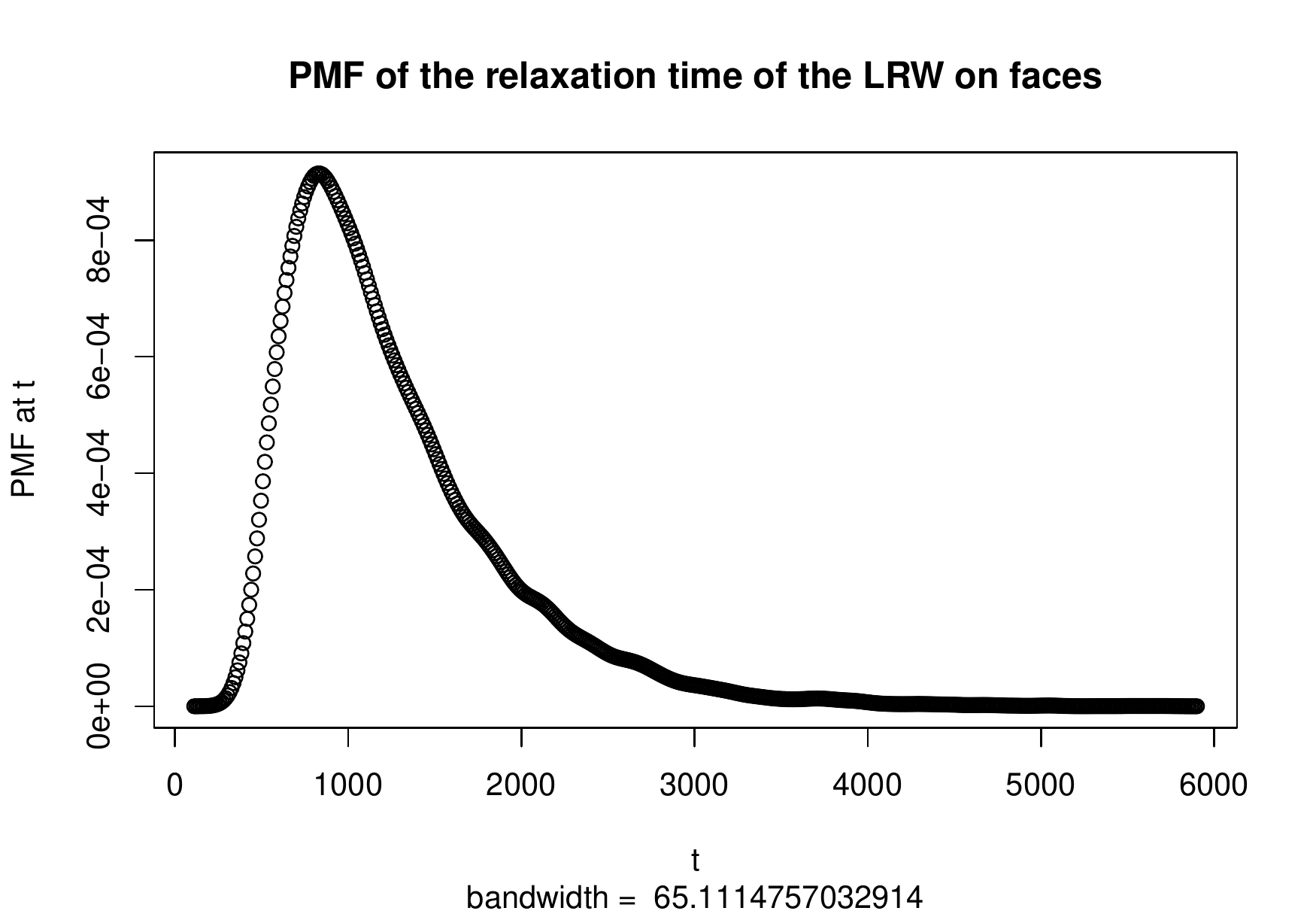}
\end{center}
\caption{Estimated PMF of various mixing times for quadrangulations with $320$ vertices. The displayed quantity is the convolution of the empirical measure by a gaussian kernel of standard deviation indicated by the “bandwidth” quantity under the graph. We note that the steepest slope is smaller than the slope of the kernel: the tail near 0 is faithfully represented. This observation is even more salliant when drawing the log-density, as in \figref{Fig_log_density_mixing_time}.}
\label{Fig_density_mixing_time}
\end{figure}

\newpage
\begin{figure}[!h]
\includegraphics[width=0.49\textwidth]{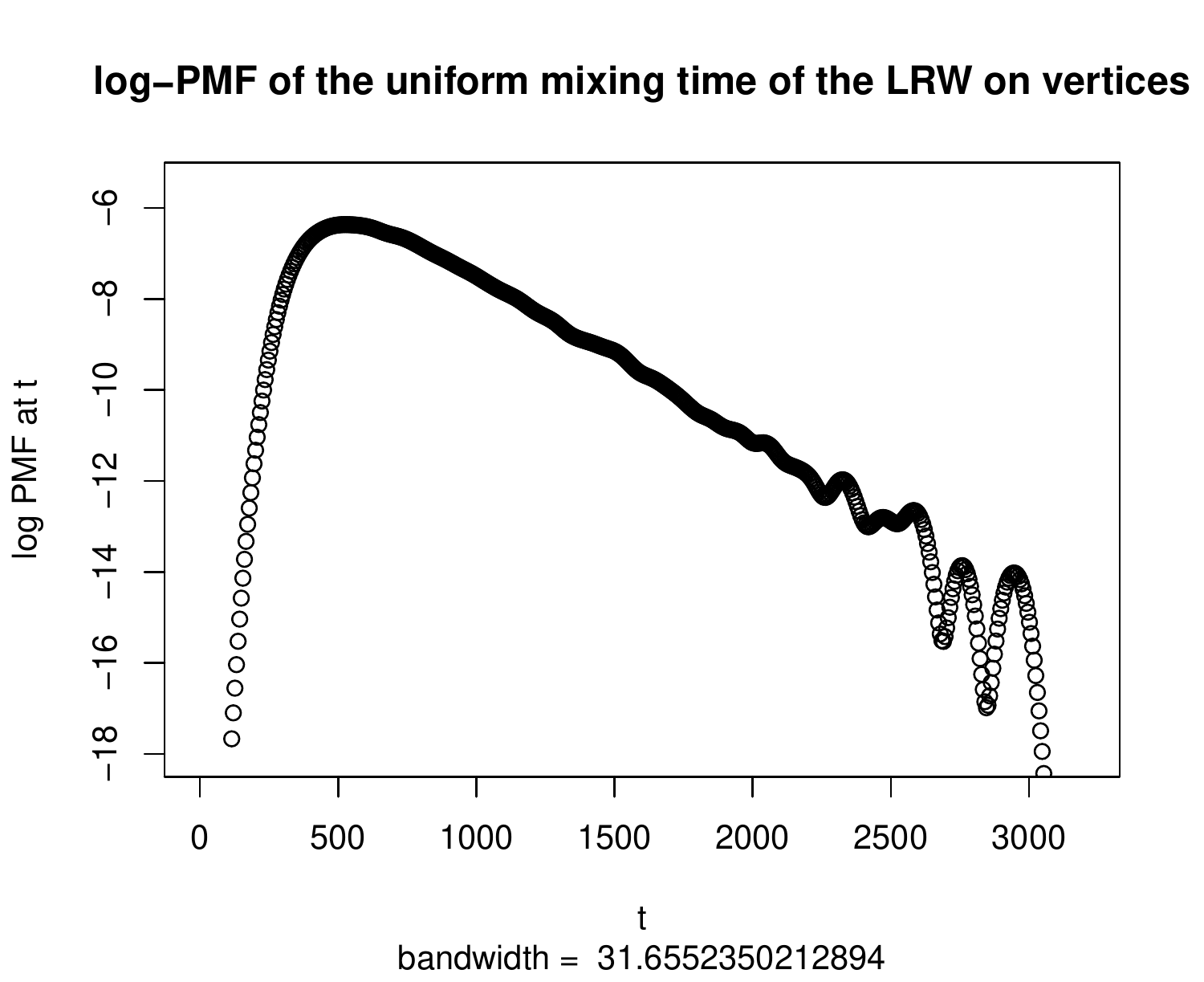}
\includegraphics[width=0.49\textwidth]{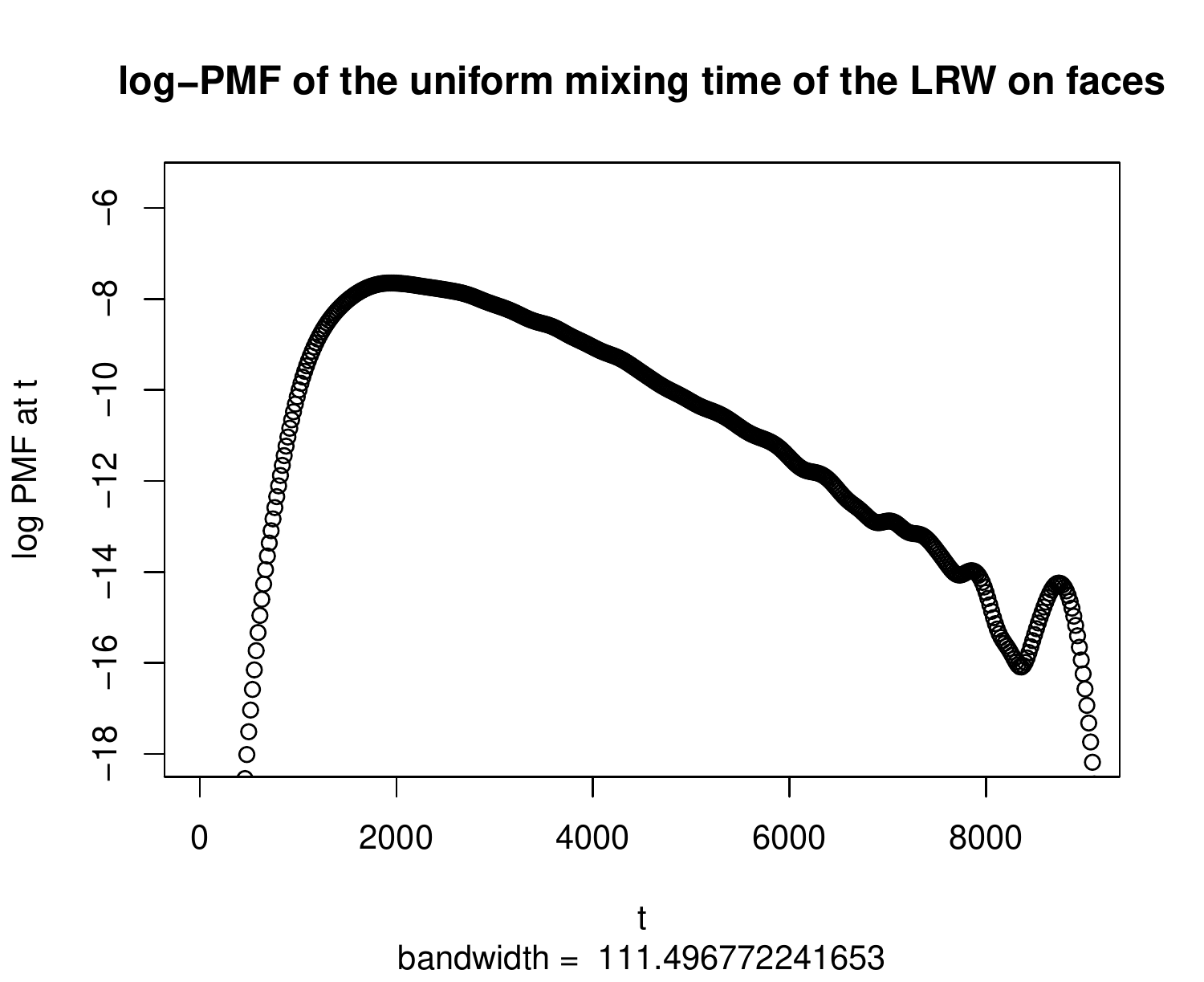}

\includegraphics[width=0.49\textwidth]{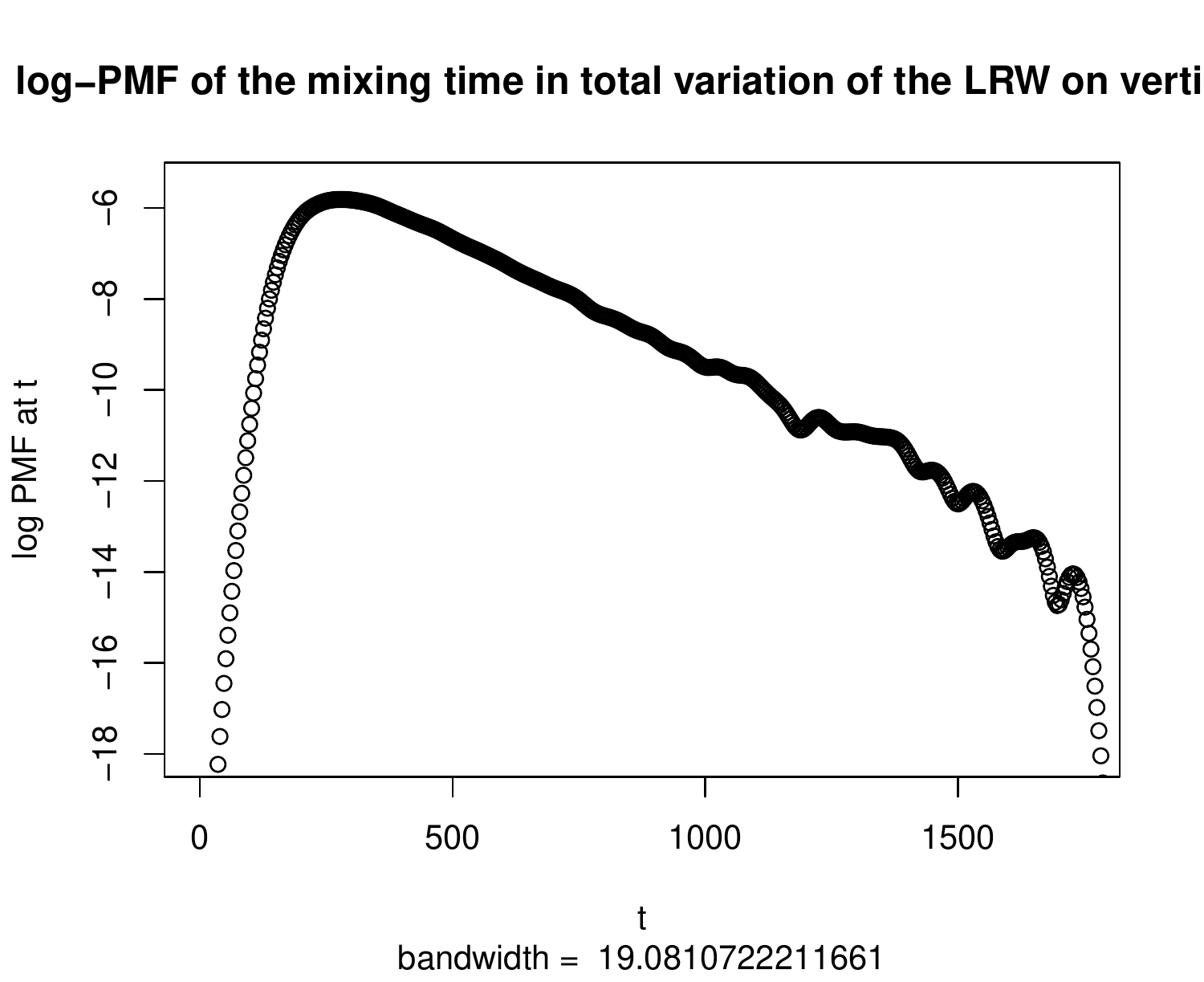}
\includegraphics[width=0.49\textwidth]{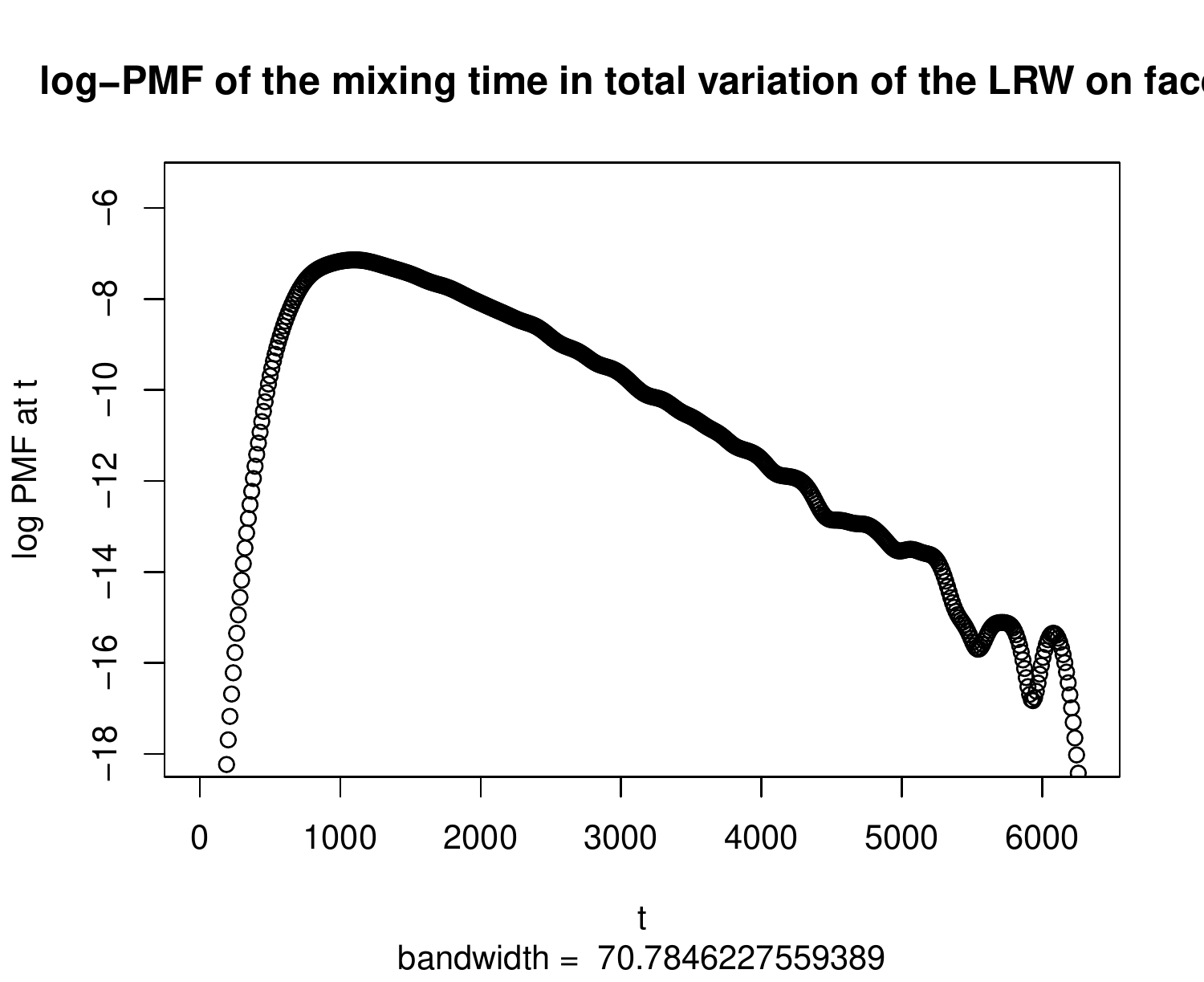}

\includegraphics[width=0.49\textwidth]{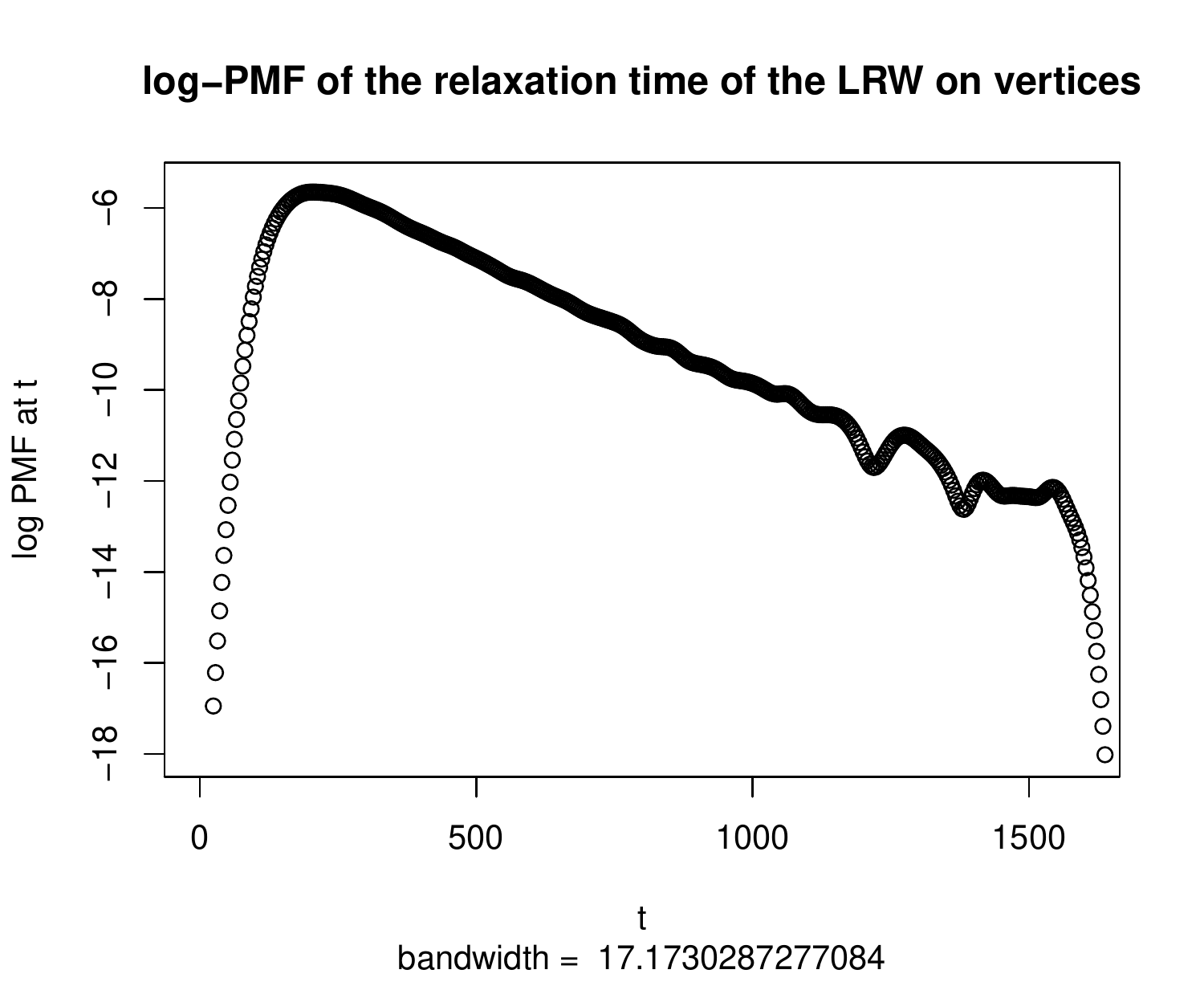}
\includegraphics[width=0.49\textwidth]{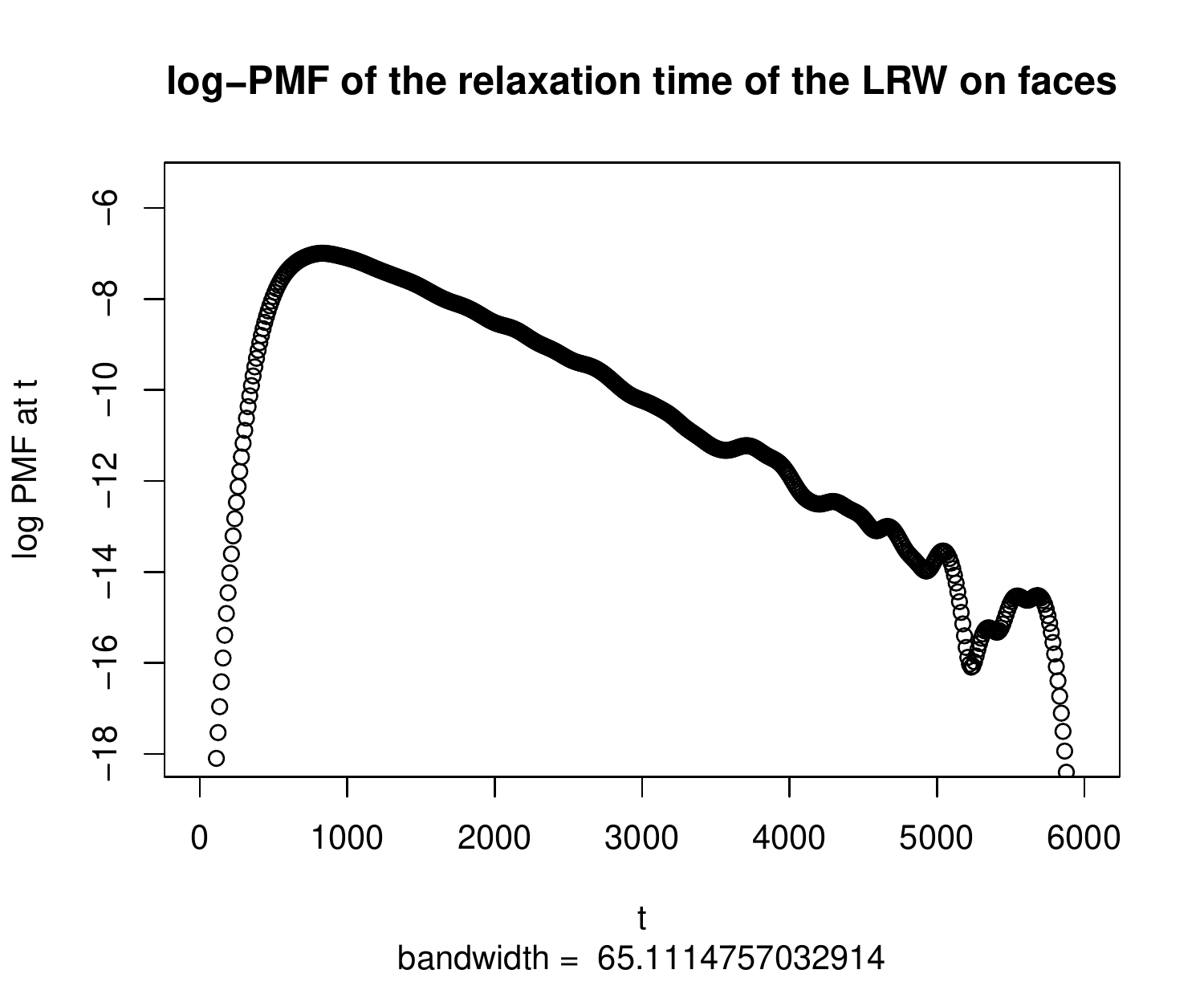}
\caption{Log of the estimated density of various mixing times for quadrangulations with $320$ vertices. The density is estimated as described in \figref{Fig_density_mixing_time}. The tail near $+\infty$ seems to be exponential.}
\label{Fig_log_density_mixing_time}
\end{figure}

\newpage
\begin{figure}[!ht]
\includegraphics[width=0.49\textwidth]{Rat-uv.pdf}
\includegraphics[width=0.49\textwidth]{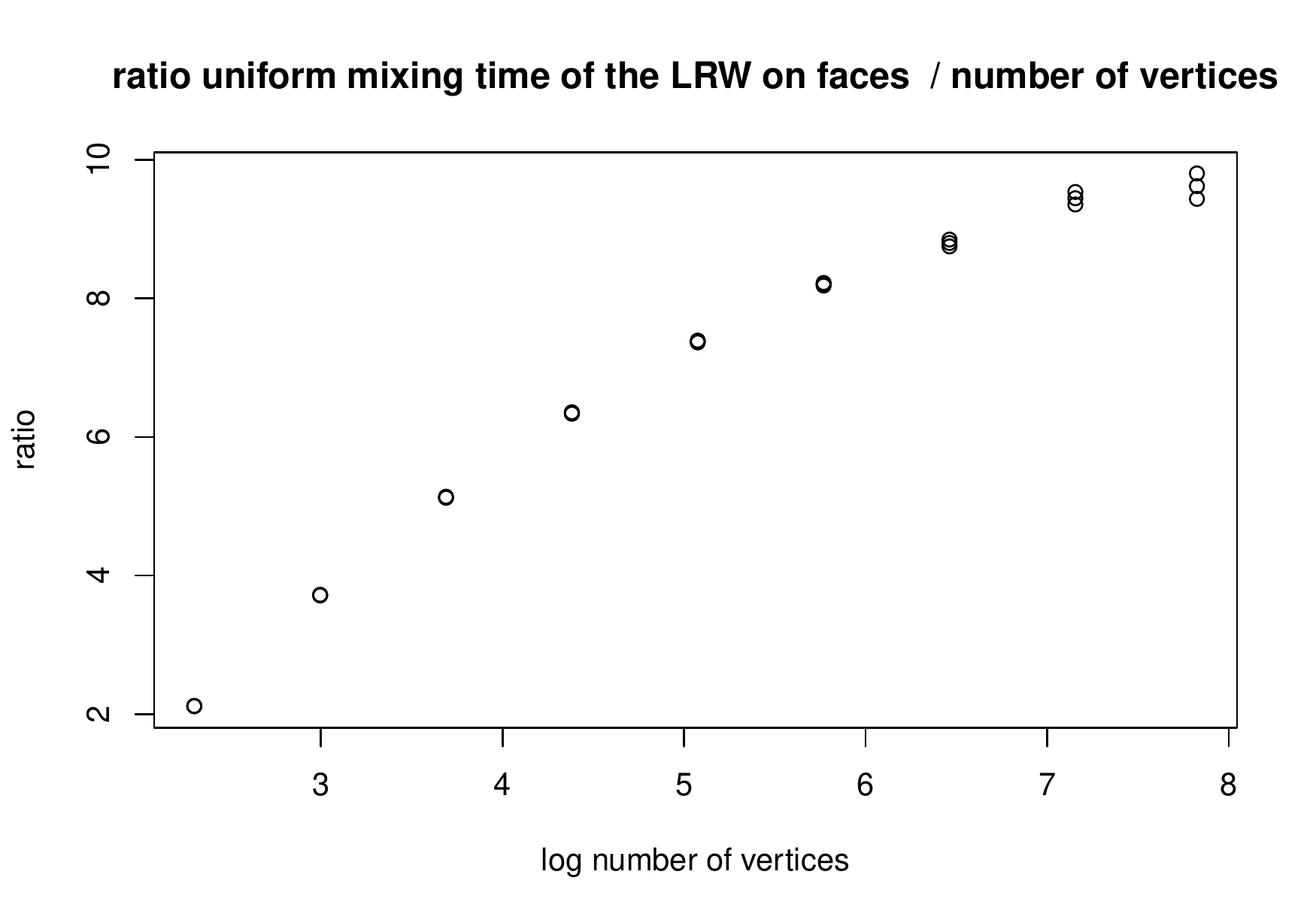}

\includegraphics[width=0.49\textwidth]{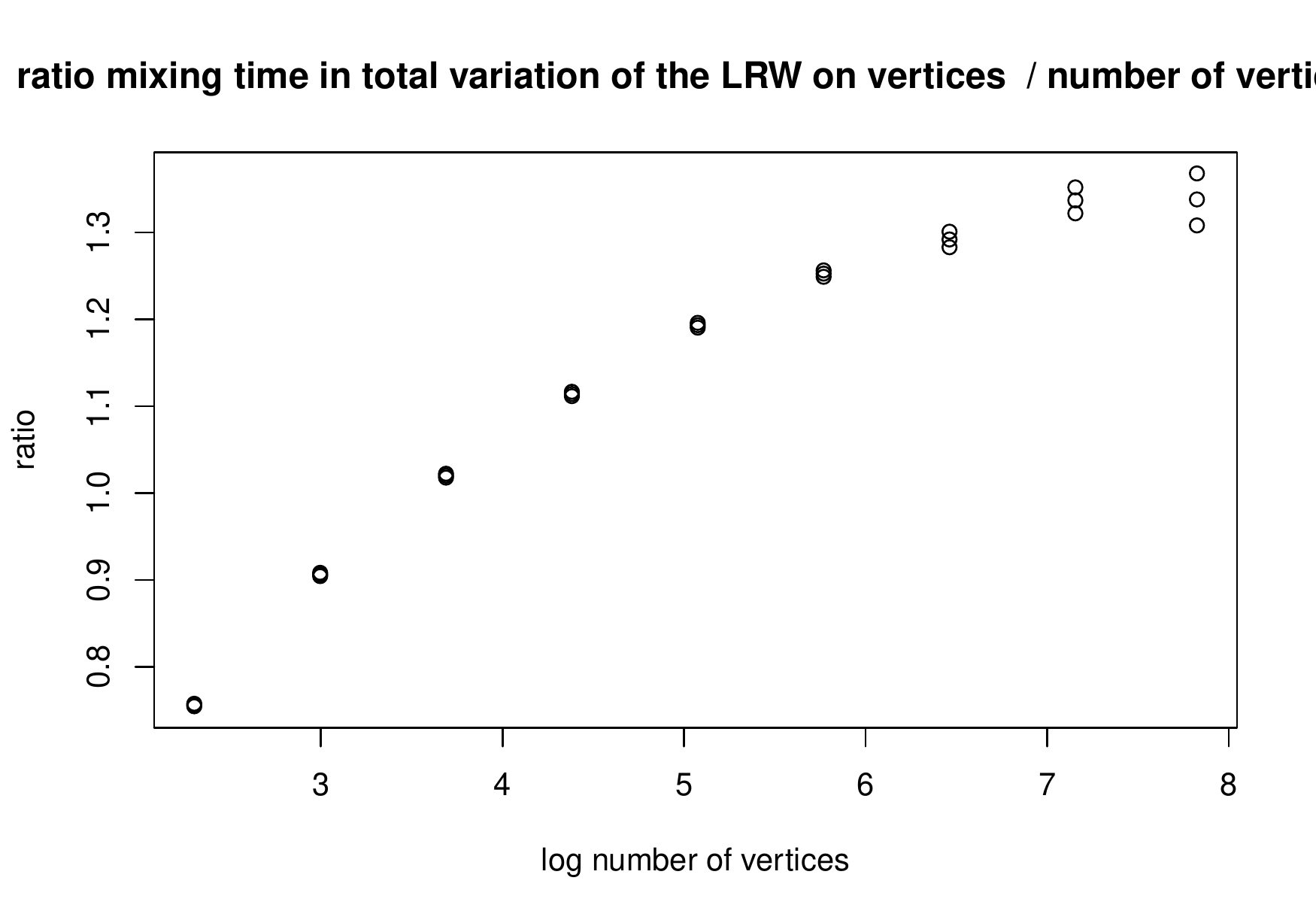}
\includegraphics[width=0.49\textwidth]{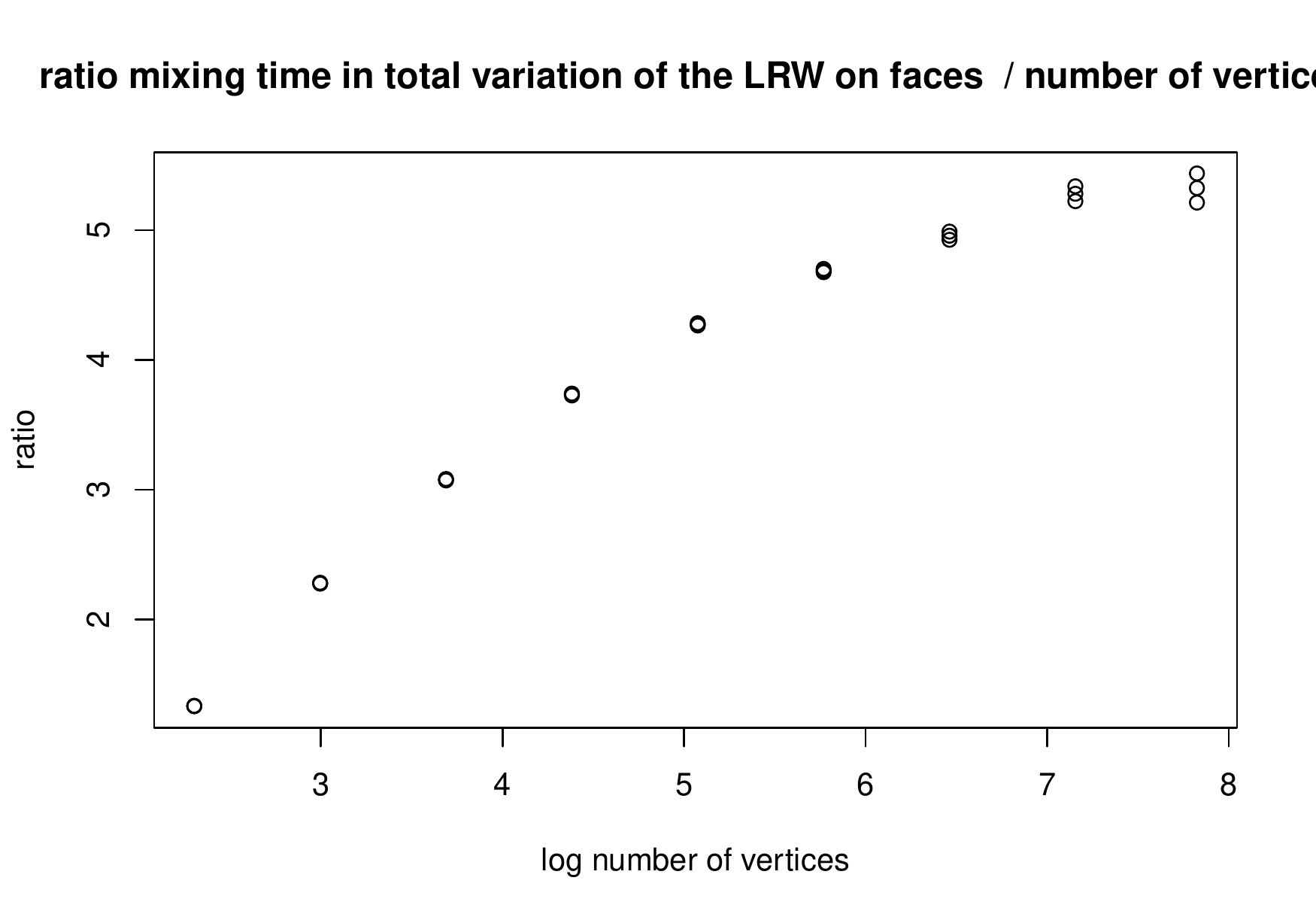}

\includegraphics[width=0.49\textwidth]{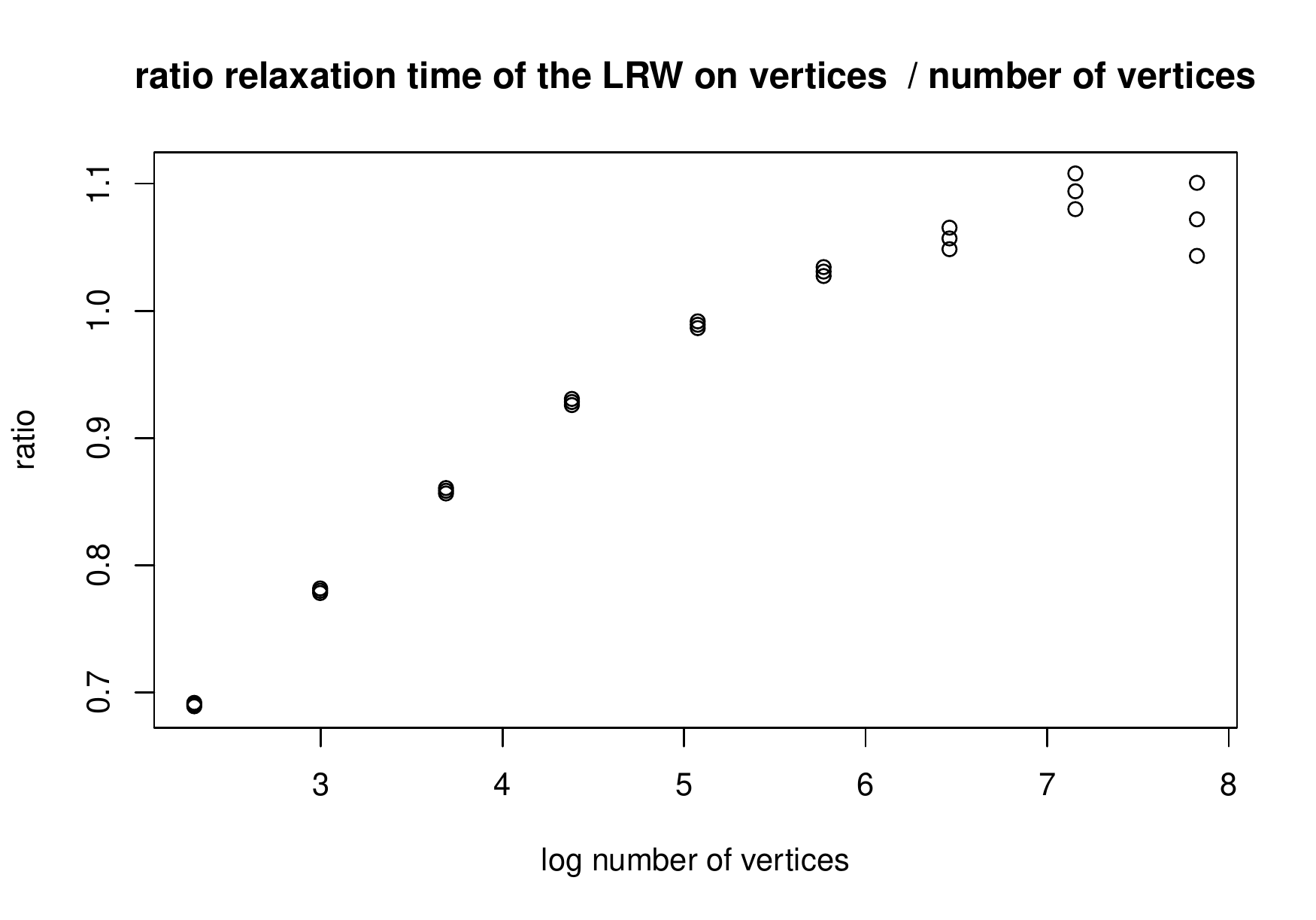}
\includegraphics[width=0.49\textwidth]{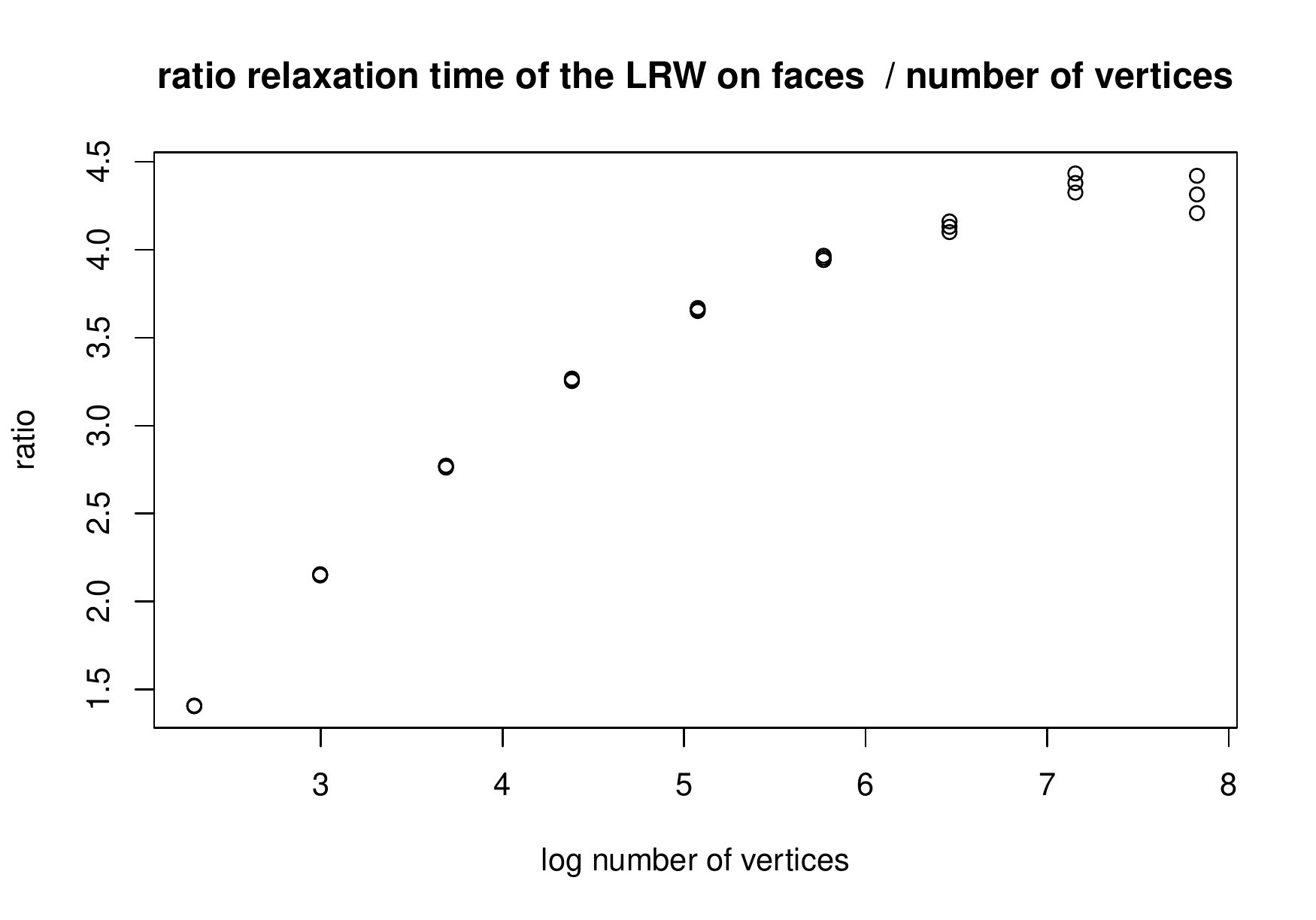}
\caption{Ratio between various mixing times and the number $n$ of vertices of the map. We display three points for each value of $n$: the central one is the empirical mean of the observations, and the top and bottom ones are at $\pm 1$ empirical standard deviation from the empirical mean. The uniform mixing times and the mixing times in total variation are taken at level $1/2$. }
\label{Fig_asymp_ratio_mixing_time}
\end{figure}

\section*{Annex: maple sheet}

\includepdf[pages=-]{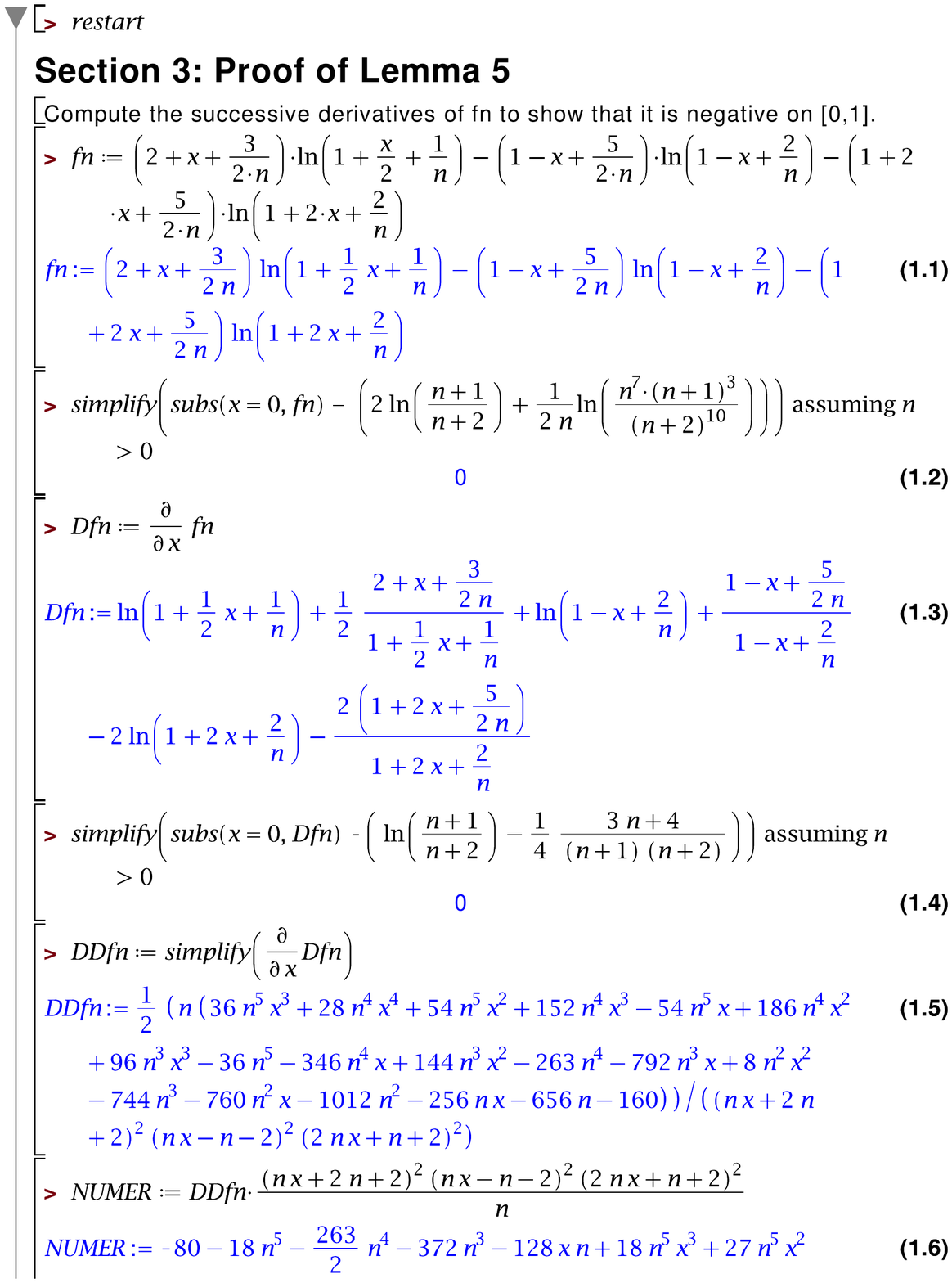}

\end{document}